\documentclass[a4paper]{amsart}

\usepackage[english]{babel}
\usepackage[utf8x]{inputenc}
\usepackage{amsmath,amssymb,amsthm}
\usepackage[table]{xcolor}
\usepackage{graphicx,enumerate,url}
\usepackage[colorinlistoftodos]{todonotes}

\usepackage{tikz}
\usetikzlibrary{calc,through,backgrounds,shapes,matrix}

\usepackage{array}
\newcolumntype{C}[1]{>{\centering\arraybackslash$}p{#1}<{$}}
\usepackage{multirow}
\usepackage{hhline}

\usepackage{hyperref}

\usepackage{url, hypcap}
\hypersetup{colorlinks=true, citecolor=darkblue, linkcolor=darkblue}

\numberwithin{equation}{section}

\usepackage{cleveref}
\crefname{conjecture}{conjecture}{conjectures}
\Crefname{conjecture}{Conjecture}{Conjectures}
\crefname{observation}{observation}{observations}
\Crefname{observation}{Observation}{Observations}
\crefname{hope}{hope}{hopes}
\Crefname{hope}{Hope}{Hopes}

\newtheorem{theorem}{Theorem}[section]
\newtheorem{proposition}[theorem]{Proposition}
\newtheorem{lemma}[theorem]{Lemma}
\newtheorem{corollary}[theorem]{Corollary}

\newtheorem{conjecture}[theorem]{Conjecture}

\theoremstyle{definition}
\newtheorem{example}[theorem]{Example}
\newtheorem{remark}[theorem]{Remark}

\newcommand{\ssm}{\smallsetminus} 

\DeclareMathOperator{\conv}{conv} 

\DeclareMathOperator{\mon}{m}

\renewcommand{\th}{\textsuperscript{th}}

\newcommand{\ZZ}{\mathbb{Z}}

\newcommand{\RR}{\mathbb{R}}

\newcommand{\Sym}{{\mathfrak{S}}}

\newcommand{\Phiplus}{\Phi^+}
\newcommand{\Phipm}{\Phi_{\geq -1}}

\newcommand{\ie}{\textit{i.e.}} 

\definecolor{darkblue}{rgb}{0,0,0.7} 
\newcommand{\darkblue}{\color{darkblue}} 

\definecolor{lightgrey}{rgb}{0.9,0.9,0.9} 
\newcommand{\lightgrey}{\cellcolor{lightgrey}} 

\definecolor{grey}{rgb}{0.5,0.5,0.5} 

\newcommand{\Dfn}[1]{\emph{\darkblue #1}} 

\newcommand{\x}{\mathbf{x}}
\newcommand{\y}{\mathbf{y}}
\renewcommand{\u}{\mathbf{u}}
\newcommand{\f}{\mathbf{f}}
\newcommand{\onevar}{\mathbf{1}}
\newcommand{\id}{{1\!\!1}} 

\newcommand{\Mpr}{\sfM^{pr}}
\newcommand{\Mex}{\sfM^{ex}}
\newcommand{\Alg}{\mathcal{A}}

\newcommand{\subwordComplex}{\mathcal{SC}} 
\newcommand{\clusterComplex}{\subwordComplex\big(\cw{c}\big)}
\newcommand{\greedy}{I_{\operatorname{g}}} 
\newcommand{\antigreedy}{I_{\operatorname{ag}}} 

\newcommand{\Fpoly}[2]{{{\sf F}_{#1}(#2)}}
\newcommand{\FpolyGen}[1]{{{\sf F}(#1)}}
\newcommand{\Xvec}[2]{{{\sf #1}({#2})}}
\newcommand{\gvec}[1]{\Xvec{g}{#1}}
\newcommand{\dvec}[1]{\Xvec{d}{#1}}
\newcommand{\cvec}[1]{\Xvec{c}{#1}}

\newcommand{\zetag}{{\zeta_{\operatorname{g}}}}
\newcommand{\zetaag}{{\zeta_{\operatorname{ag}}}}

\newcommand{\sq}[1]{{\rm #1}} 
\newcommand{\Q}{\sq{Q}} 
\newcommand{\q}{\sq{q}} 
\newcommand{\s}{\sq{s}} 

\newcommand{\sref}{\mathcal{S}}

\newcommand{\sfr}{{\sf r}} 
\newcommand{\sfR}{{\sf R}} 
\newcommand{\sfw}{{\sf w}} 
\newcommand{\sfW}{{\sf W}} 
\newcommand{\sfB}{{\sf B}} 

\newcommand{\wo}{{w_\circ}}

\renewcommand{\c}{\sq{c}}
\newcommand{\sw}[2]{\sq{#1}(\sq{#2})} 
\newcommand{\cwo}[1]{\sw{\wo}{#1}} 
\newcommand{\cw}[1]{\sq{#1}\cwo{#1}} 

\newcommand{\brickVector}{\sfB} 
\newcommand{\brickPolytope}{\mathcal{B}} 
\newcommand{\Root}[2]{{\sfr}(#1,#2)} 
\newcommand{\Roots}[1]{{\sfR}(#1)} 
\newcommand{\Weight}[2]{{\sfw}(#1,#2)} 
\newcommand{\Weights}[1]{{\sfW}(#1)} 

\newcommand{\IntervalWeight}[1]{{\operatorname{Int}\big(\Weight{\greedy}{#1},\ \Weight{\antigreedy}{#1}\big)}}
\newcommand{\IntervalRoot}[1]{{\operatorname{Int}\big(0,\ \Root{\greedy}{#1}\big)}}

\newcommand{\coRoot}[2]{{\sfr^\vee}(#1,#2)} 
\newcommand{\coRoots}[1]{{\sfR^\vee}(#1)} 
\newcommand{\coWeight}[2]{{\sfw^\vee}(#1,#2)} 
\newcommand{\coWeights}[1]{{\sfW^\vee}(#1)} 

\newcommand{\Newton}[1]{{\operatorname{Newton}\big( #1 \big)}}

\newcommand{\set}[2]{\left\{ #1 \;:\; #2 \right\}} 
\newcommand{\bigset}[2]{\big\{ #1 \;:\; #2 \big\}} 
\newcommand{\bigmultiset}[2]{\big\{\!\!\big\{ #1 \;:\; #2 \big\}\!\!\big\}} 

\newcommand{\Mink}{{\sum}}

\newcommand{\dotprod}[2]{\left\langle \, #1 \,\middle|\, #2 \, \right\rangle} 
\newcommand{\bigdotprod}[2]{\big\langle \, #1 \,\big|\, #2 \, \big\rangle} 

\newcommand{\wordprod}[2]{\Pi{#1}_{#2}} 

\newcommand{\sfM}{{\sf M}} 
\newcommand{\mutmatrix}{\widetilde\sfM} 

\newcounter{intege}
\newcommand{\ngon}[9]{ 

  \foreach \t in {1,...,#1} {
    \coordinate (#7\t) at ($#2+(#8-\t*360/#1:#3)$);
  }
  \foreach \x/\y/\z in {#4}{
    \draw[\z,shorten <=#5pt, shorten >=#5pt] {(#7\x)--(#7\y)};
  }
  \foreach \x/\y/\z/\o in {#9}{
    \coordinate (#7_tmp)  at ($0.33*(first\x) + 0.33*(first\y) + 0.33*(first\z)$);
    \coordinate (#7_tmp1) at ($0.2*(#7_tmp) + 0.8*(first\x)$);
    \coordinate (#7_tmp2) at ($0.2*(#7_tmp) + 0.8*(first\y)$);
    \coordinate (#7_tmp3) at ($0.2*(#7_tmp) + 0.8*(first\z)$);
    \fill[fill opacity=\o] (#7_tmp1) -- (#7_tmp2) -- (#7_tmp3);
  }
  \setcounter{intege}{1}
  \pgfmathsetcounter{intege}{1}
  \foreach \object in {#6}{
    \node[inner sep=0pt] at (#7\theintege) {\object};
    \pgfmathsetcounter{intege}{\theintege+1}
    \setcounter{intege}{\theintege}
  }
}

\title[Finite type cluster algebras from subword complexes]{Towards a uniform subword complex description of acyclic finite type cluster algebras}

\author[S.~B.~Brodsky]{Sarah~B.~Brodsky$^*$}
\address[S.~B.~Brodsky]{Department of Mathematics, Technische Universit\"{a}t Berlin, Germany}
\email{brodsky@math.tu-berlin.de}
\thanks{$^*$ Supported by the European Research Council grant SHPEF awarded to
Olga Holtz and the Berlin Mathematical School}

\author[C.~Stump]{Christian Stump$^{\dagger}$}
\address[C.~Stump]{Institut f\"ur Mathematik, Freie Universit\"at Berlin, Germany}
\email{christian.stump@fu-berlin.de}
\thanks{$^{\dagger}$Supported by the German Research Foundation DFG, grants STU 563/2-2 ``Coxeter-Catalan combinatorics" and ``Combinatorial and geometric structures for reflection groups and groupoids'' within SPP 1489.}

\thanks{\textbf{Acknowledgements:} The second author thanks Vincent Pilaud, Nathan Reading, and Hugh Thomas for various discussions about finite type cluster algebras and their combinatorics.}
\date{\today}

\begin{document}

\begin{abstract}
  It has been established in recent years how to approach acyclic cluster algebras of finite type using subword complexes.
  In this paper, we continue this study by describing the $c$- and $g$-vectors, and by providing a conjectured description of the Newton polytopes of the $F$-polynomials.
  In particular, we show that this conjectured description would imply that finite type cluster complexes are realized by the duals of the Minkowski sums of the Newton polytopes of either the $F$-polynomials, or of the cluster variables, respectively.
\end{abstract}

\maketitle

\setcounter{tocdepth}{1}
\tableofcontents

\section{Introduction}

Let $(W,\sref)$ be a finite crystallographic Coxeter system of rank~$n$ with simple system~$\sref$, and let~$c \in W$ be a (standard) Coxeter element for $(W,\sref)$; \ie~$c = s_1\cdots s_n$ is the product of all elements in~$\sref$ in some order.
Let $A = (a_{st})_{s,t \in \sref}$ be a crystallographic Cartan matrix for $(W,\sref)$; \ie~an integral matrix such that $a_{ss} = 2$, $a_{st} \le 0$, $a_{st}a_{ts} = 4 \cos^2(\frac{\pi}{m_{st}})$ and ${a_{st} = 0 \Leftrightarrow a_{ts} = 0}$ for all~$s \ne t \in \sref$ where $m_{st}$ is the order of $st$ in~$W$, and let $\Delta \subseteq \Phiplus \subseteq \Phipm \subseteq \Phi$ be the resulting root system with simple roots $\Delta = \set{\alpha_s}{s \in \sref}$, positive roots~$\Phiplus = \Phi\cap\ZZ_{\geq 0}\Delta$, almost positive roots~$\Phipm = \Phiplus \sqcup -\Delta$, and root system $\Phi=\set{w(\alpha)}{w \in W, \alpha\in\Delta}$.
For convenience, we set~$\alpha_i := \alpha_{s_i}$.
As for $s \neq t$, we have that $a_{st} = 0$ if and only if $st = ts$, we think of the Coxeter element~$c$ as an acyclic orientation of the Dynkin diagram by orienting an edge $s \rightarrow t$ if~$s$ comes before~$t$ in the given reduced word~$\c = \s_1 \s_2 \cdots \s_n$ (or, equivalently, in all reduced words for~$c$).
It is indeed the case that this mapping yields a one-to-one correspondence between Coxeter elements and acyclic orientations of the Dynkin diagram, and that the reduced words for a Coxeter element are given by all linear extensions of this orientation.
The reason we work with the particular word $\c = \s_1\cdots\s_n$ is that it will later simplify several notations, both in the indexing of variables (see the following paragraph), and in the definition of subword complexes (see the beginning of \Cref{sec:results}).

It is well established how to associate to this data an initial seed of a cluster algebra of finite type with principal coefficients.
We refer to \Cref{sec:proofextendedpart} and also to~\cite{FZ2007,FZ20022,FZ2003} for the needed background on cluster algebras.
For a given such orientation of the Dynkin diagram, define the skew-symmetrizable matrix $\sfM_c = (b_{st})_{s,t \in \sref}$ by
\begin{align}
  b_{st} = \begin{cases}
             -a_{st} & \text{ if } s \rightarrow t,\\
             a_{st}  & \text{ if } s \leftarrow t,\\
             0       & \text{ else}.
           \end{cases} \label{eq:bmatrix}
\end{align}
The initial cluster seed is then given by $\big(\mutmatrix_c, \x, \y \big)$, where $\mutmatrix_c$ is the $(2n \times n)$-matrix, known as the exchange matrix, $\begin{bmatrix} \sfM_c \\ \id_n \end{bmatrix}$ with principal part~$\sfM_c$ and extended part~$\id_n$ an identity matrix, $\x = (x_1,\ldots,x_n)$ are the cluster variables (the cluster of the seed), and $\y = (y_1,\ldots,y_n)$ are the frozen variables (the coefficients of the seed).
In the present context, one should think of the variables~$x_k$ and~$y_k$ as being indexed by~$\alpha_k$ for $1 \leq k \leq n$, so they are in particular indexed in a way that is consistent with the order of the simple reflections in the given Coxeter element~$c$.
Let $\Alg(W,c) := \Alg(\mutmatrix_c)$ be the cluster algebra generated from this initial seed.

\medskip

It is known that every cluster variable~$u(\x,\y) \in \Alg(W,c)$ lives inside the ring $\ZZ[x_1^{\pm 1},\ldots,x_n^{\pm 1};y_1,\ldots,y_n]$, \ie, $u(\x,\y) = p(\x,\y)/\mon(\x)$ where~$p(\x,\y)$ is a polynomial in~$\x,\y$ with integer coefficients and $\mon(\x)$ in a monomial in~$\x$, see~\cite[Proposition~3.6]{FZ2007}.
The \Dfn{$d$-vector} $\dvec{u}$ of~$u(\x,\y)$ is the exponent vector of the denominator monomial $\mon(\x)$, \ie, $\dvec{u} = (d_1,\ldots,d_n)$ for $\mon(\x) = x_1^{d_1}\cdots x_n^{d_n}$ and should be thought of as a vector in the basis $\Delta$, \ie, $\dvec{u} = d_1\alpha_{1} + \ldots + d_n\alpha_{n}$.
Under this identification, it is shown in~\cite[Theorem~1.9]{FZ2003}, that the map $u \mapsto \dvec{u}$ is a bijection between cluster variables in $\Alg(W,c)$ and almost positive roots~$\Phipm$, and that furthermore, $\dvec{u} = -\alpha_i \in -\Delta$ if and only if $u(\x,\y) = x_i$.
We will regularly use this bijection in indexing objects.
For example, we denote by $\Fpoly{u}{\y} = \Fpoly{\beta}{\y} = u(\onevar,\y) = p(\onevar,\y)$ the \Dfn{$F$-polynomial} associated to~$u(\x,\y) \in \Alg(W,c)$ and to $\beta \in \Phipm$ with $\dvec{u} = \beta$.
As $\Fpoly{\beta}{\y} = 1$ for $\beta \in -\Delta$, one usually considers $F$-polynomials only for positive roots $\beta \in \Phiplus$.
We also denote by $\gvec{u} = \gvec{\beta} = (g_1,\ldots,g_n)$ the \Dfn{$g$-vector} given by the exponent vector of~$u(\x,0)$.
For reasons that will become clear later, we consider the $g$-vector to live inside the weight space, \ie, $\gvec{u} = g_1\omega_1+\cdots+g_n\omega_n$.

\medskip

For the course of this paper, it will be natural to consider the vector notation in the exchange matrices inside the root space:
we think of any exchange matrix $\mutmatrix = \begin{bmatrix} \Mpr \\ \Mex \end{bmatrix}$ of a cluster seed of $\Alg(W,c)$ with cluster $(u_1,\ldots,u_n)$ as being indexed as follows.
Row and column~$i$ of~$\Mpr$ are both indexed by the almost positive root $\dvec{u_i}$.
Equally, column~$i$ of~$\Mex$ is indexed by this almost positive root, while row~$i$ of $\Mex$ is indexed by the simple root~$\alpha_i$.
The \Dfn{$c$-vector} $\cvec{u} = \cvec{\beta}$ with $\beta = \dvec{u}$ \emph{inside this cluster seed} is then given by the column vector of~$\Mex$ in the column indexed by the almost positive root~$\beta$, written as a linear combination of the simple roots,
\[
  \cvec{u} = \cvec{\beta} = [\Mex]_{\alpha_1,\beta} \alpha_1 + \ldots + [\Mex]_{\alpha_n,\beta} \alpha_n,
\]
where we emphasize that this definition does not only depend on the variable $u(x,y)$ but on the actual cluster seed.

\medskip

Every cluster seed is uniquely determined by its cluster, and the \Dfn{cluster complex} of $\Alg(W,c)$ is the simplicial complex with ground set being the set of cluster variables, and with facets being the clusters.
Cluster complexes of finite type with the initial seed coming from a bipartite Coxeter element  (\ie, those where every vertex in the corresponding orientation of the Dynkin diagram is either a sink or a source) were studied and completely described in terms of compatibility of $d$-vectors in~\cite{FZ2003}.
Polytopal realizations of the cluster complex of type $\Alg(W,c)$ were first obtained by F.~Chapoton, S.~Fomin, and A.~Zelevinsky in~\cite{CFZ2002} for bipartite Coxeter elements, and by C.~Hohlweg, C.~Lange, and H.~Thomas in~\cite{HLT2011} for general Coxeter elements.

\medskip

Despite the nice combinatorial descriptions of the cluster complex and its polytopal realization in terms of the corresponding root system given by sending a cluster variable to its denominator vector, to the best of our knowledge there has not been any successful attempt to describe the numerator of the cluster variables from that perspective.
In particular, no explicit construction of the cluster variables for general finite type cluster algebras is known that does not use the defining iterative procedure (which we recall in 
\Cref{sec:proofextendedpart}).

\smallskip
\begin{center}
  \emph{The main aim of this paper is to start the program of describing cluster variables\\
  in finite type cluster algebras in terms of combinatorial data from root systems.}
\end{center}
\smallskip

Towards such a description, we follow the recently introduced subword complex approach to finite type cluster algebras.
These subword complexes were originally considered by A.~Knutson and E.~Miller in the context of Gr\"obner geometry of Schubert varieties in~\cite{KM2005,KM2004}.
Their appearance in the context of finite type cluster algebras was established by C.~Ceballos, J.-P.~Labb\'e, V.~Pilaud and the second author in various collaborations.
In particular, they establish
\begin{itemize}
  \item a description of the cluster complex of the cluster algebra $\Alg(W,c)$ in~\cite[Theorem~2.2]{CLS2014},
  \item a vertex and facet description of its polytopal realization in~\cite[Theorem~6.4]{PS2015},
  \item a proof that the barycenter of this realization equals the barycenter of the corresponding permutahedron in~\cite[Theorem~1.1]{PS20152},
  \item an explicit description of the principal parts of the exchange matrices of the clusters~\cite[Theorem~6.20]{PS2015}, and
  \item an explicit description of the $d$-vectors with respect to any initial seed (including cyclic seeds) in~\cite[Corollary~3.4]{CP2015}.
\end{itemize}

In the present paper, we extend this viewpoint by providing the following two constructions in terms of subword complexes:
\begin{enumerate}
  \item We show in \Cref{thm:extendedpart} that the $c$-vectors of a cluster seed of the cluster algebra $\Alg(W,c)$ are given by the root configuration defined in~\eqref{eq:rootconfiguration}, and deduce in \Cref{cor:extendedpart} that the $g$-vectors are given by the weight configuration defined in~\eqref{eq:weightconfiguration}.

  \item We start the development of describing the $F$-polynomials for $\Alg(W,c)$ in \Cref{conj:newtonpolytope,conj:latticepoints} by conjecturally providing all their monomials and in particular their Newton polytopes.
  Both are then proven in type~$A_n$ and and small ranks including all exceptional types in \Cref{thm:main1,thm:main2}.
  A combinatorial description of the coefficients of the monomials of the $F$-polynomials would therefore be the last step to provide a complete combinatorial description of the cluster algebra as it is well known how to recover the cluster variables from the the $g$-vectors and the $F$-polynomials as we recall in \Cref{prop:clustervariables}.
\end{enumerate}

Two further remarks about previous work are in order.

\begin{remark}
  As we will later use, R.~Schiffler gave in~\cite{S2008} an explicit description of the cluster variables of type~$A_n$ via \emph{$T$-paths} on triangulations of the regular $(n+3)$-gon, and G.~Musiker and R.~Schiffler generalized that description in~\cite{MS2010} to cluster variables for cluster algebras associated to \emph{unpunctured surfaces} with arbitrary coefficients.
  Together with L.~Williams, they extended in~\cite{MSW2011} the results also to arbitrary surfaces, allowing punctures.
  We refer to \Cref{sec:proofmain1} for further details.
  We finally note that their combinatorial model for punctured surfaces can also be used to combinatorially describe the $F$-polynomials for type~$D$ cluster complexes.
\end{remark}

\begin{remark}
  N.~Reading and D.~Speyer provide in~\cite{RS2016} a general combinatorial framework for acyclic cluster algebras to obtain information about exchange matrices, and $c$- and $g$-vectors.
  See \Cref{rem:RS1,rem:RS2} for further details.
  As we will discuss in those remarks, both approaches are closely related.
  The two main differences currently are that our approach has not been extended beyond finite types (see also~\cite{RS2015}), while their approach only uses (their versions of) the root and the coroot configurations, but does not (seem to) provide information about $F$-polynomials and cluster variables.
\end{remark}

We finish the introduction by providing the following well understood running example of type~$A_2$ with principal coefficients.
We will use this later to visualize the close relationship with the combinatorics of subword complexes.
\textbf{This example has to be treated with caution} as it does not show several difficulties that appear in types other than~$A_n$, as here we have that Newton polytopes of $F$-polynomials have no inner lattice points, and all their monomials appear with coefficient~$1$.

\begin{example}\label{ex:clusterA2}
  The five cluster seeds and the dual cluster complex are given by
  \begin{center}
    \begin{tikzpicture}[scale=0.5]
      \coordinate (allo) at (0,0);
      \ngon{5}{(allo)}{5}
        {1/2/black,2/3/black,3/4/black,4/5/black,5/1/black}{25}{,,,,}{first}{162}{}

        \matrix [nodes={rectangle,text centered,inner sep=2pt}, matrix of nodes, font=\scriptsize] at (first1)
          { 0 & 1 \\ -1 & 0 \\ 1 & 0 \\ 0 & 1 \\ };
        \matrix [nodes={rectangle,text centered,inner sep=2pt}, matrix of math nodes] at ($(first1)+(2.2,0)$)
        { \big\{x_1,x_2\big\} \\ \big\{y_1,y_2\big\} \\ };

        \matrix [nodes={rectangle,text centered,inner sep=2pt}, matrix of nodes, font=\scriptsize] at (first2)
          { 0 & -1 \\ 1 & 0 \\ -1 & 1 \\ 0 & 1 \\ };
        \matrix [nodes={rectangle,text centered,inner sep=2pt}, matrix of math nodes] at ($(first2)+(2.9,0)$)
          {\big\{\frac{x_{2} + y_{1}}{x_{1}}, x_{2}\big\} \\ \big\{\frac{1}{y_{1}}, y_{1} y_{2}\big\} \\ };

        \matrix [nodes={rectangle,text centered,inner sep=2pt}, matrix of nodes, font=\scriptsize] at ($(first5)-(4,0)$)
          { 0 & -1 \\ 1 & 0 \\ 1 & 0 \\ 0 & -1 \\ };
        \matrix [nodes={rectangle,text centered,inner sep=2pt}, matrix of math nodes] at ($(first5)-(1,0)$)
          {\big\{x_{1}, \frac{x_{1} y_{2} + 1}{x_{2}}\big\} \\ \big\{y_{1}, \frac{1}{y_{2}}\big\} \\ };

        \matrix [nodes={rectangle,text centered,inner sep=2pt}, matrix of nodes, font=\scriptsize] at ($(first4)-(6.5,0)$)
          { 0 & 1 \\ -1 & 0 \\ -1 & 0 \\ 0 & -1 \\ };
        \matrix [nodes={rectangle,text centered,inner sep=2pt}, matrix of math nodes] at ($(first4)-(2,0)$)
          {\big\{\frac{x_{1} y_{1} y_{2} + x_{2} + y_{1}}{x_{1} x_{2}}, \frac{x_{1} y_{2} + 1}{x_{2}}\big\} \\ \big\{\frac{1}{y_{1}}, \frac{1}{y_{2}}\big\} \\ };

        \matrix [nodes={rectangle,text centered,inner sep=2pt}, matrix of nodes, font=\scriptsize] at ($(first3)-(0,0)$)
          { 0 & -1 \\ 1 & 0 \\ -1 & 0 \\ -1 & 1 \\ };
        \matrix [nodes={rectangle,text centered,inner sep=2pt}, matrix of math nodes] at ($(first3)+(4.5,0)$)
          {\big\{\frac{x_{1} y_{1} y_{2} + x_{2} + y_{1}}{x_{1} x_{2}}, \frac{x_{2} + y_{1}}{x_{1}}\big\} \\ \big\{\frac{1}{y_{1} y_{2}}, y_{2}\big\} \\ };
    \end{tikzpicture}
  \end{center}

  Observe that between the two clusters $\big\{\frac{x_{2} + y_{1}}{x_{1}}, x_{2}\big\}$ and $\big\{\frac{x_{1} y_{1} y_{2} + x_{2} + y_{1}}{x_{1} x_{2}}, \frac{x_{2} + y_{1}}{x_{1}}\big\}$ we switched the position of the common variable in the sense that the two columns and the first two rows of the mutation matrices switched.
  This is unavoidable and has to be done within this 5-cycle. As mentioned, we prefer to think of the columns and rows being indexed by almost positive roots and simple roots rather than such a linear listing.
  The cluster variables, their $d$- and $g$-vectors, and $F$-polynomials are thus given by
  \[
    \begin{array}{|c|c|c|c|}
      \hline
      &&&\\[-9pt]
      u(\x,\y) & \dvec{u}\in \Phipm & \gvec{u} & \Fpoly{u}{\y} \\[2pt]
      \hline
      &&&\\[-9pt]
      x_1 = \frac{1}{x_1^{-1}} & -\alpha_1 & \alpha_1 & 1 \\[5pt]

      x_2 = \frac{1}{x_2^{-1}} & -\alpha_2 & \alpha_2 & 1 \\[5pt]

      \frac{x_{2} + y_{1}}{x_{1}} & \alpha_1 & \alpha_2-\alpha_1 & y_1 + 1 \\[5pt]

      \frac{x_{1} y_{1} y_{2} + x_{2} + y_{1}}{x_{1} x_{2}} & \alpha_1+\alpha_2 & -\alpha_1 & y_1y_2 + y_1 + 1 \\[5pt]

      \frac{x_{1} y_{2} + 1}{x_{2}} & \alpha_2 & -\alpha_2 & y_2 + 1 \\[5pt]
      \hline

    \end{array}
  \]

\end{example}

\section{Definitions and main results}
\label{sec:results}

We start this section with recalling notions from finite root systems and from the theory of subword complexes and their relations to the cluster algebra $\Alg(W,c)$.
We refer to~\cite{PS2015} for a detailed treatment of these notions and further background.

\medskip

Consider the finite crystallographic Coxeter system~$(W,\sref)$ acting essentially on a Euclidean vector space~$V$ of dimension~$n$ with inner product~$\dotprod{\cdot}{\cdot}$, with \Dfn{simple roots} $\Delta = \set{\alpha_s}{s \in \sref}$ and \Dfn{simple coroots}~$\Delta^\vee = \set{\alpha_s^\vee}{s \in \sref}$.
We then have that $\alpha_s^\vee = 2\alpha_s / \dotprod{\alpha_s}{\alpha_s}$, that the \Dfn{crystallographic Cartan matrix} $A = (a_{st})_{s,t \in \sref}$ is given by $a_{st} = \dotprod{\alpha_t}{\alpha_s^\vee}$, and that $s(\alpha_t) = \alpha_t - a_{st} \alpha_s$.
The \Dfn{fundamental weights} $\nabla = \set{\omega_s}{s \in \sref}$ and \Dfn{fundamental coweights} $\nabla^\vee = \set{\omega_s^\vee}{s \in \sref}$ are the bases dual to the simple coroots and to the simple roots, respectively.
That is,
\[
  \dotprod{\omega_s}{\alpha_t^\vee} = \dotprod{\alpha_s}{\omega_t^\vee} = \delta_{s=t}.
\]
It is then easy to check that
\begin{align}
  \alpha_s = \sum_{t \in \sref}a_{ts}\omega_t, \quad \alpha_s^\vee = \sum_{t \in \sref}a_{st}\omega_t^\vee, \label{eq:rootweight}
\end{align}
and that moreover, $s(\omega_t) = \omega_t - \delta_{s=t}\alpha_s$ for $s \in \sref$.
Given the Coxeter element~$c$ with the fixed reduced word $\c = \s_1\cdots\s_n$, we will often write $\alpha_i = \alpha_{s_i}$, $\alpha_i^\vee = \alpha_{s_i}^\vee$, $\omega_i = \omega_{s_i}$, and $\omega_i^\vee = \omega_{s_i}^\vee$ to avoid double indices.
Denote by $L = \ZZ\Delta$ the $\ZZ$-lattice spanned by $\Delta$, by~$L^+$ the nonnegative span $\ZZ_{\geq 0}\Delta$, and by~$L^- = -L^+$ the nonpositive span.
We call $\beta \in L$ \Dfn{sign-coherent} if $\beta \in L^+ \sqcup L^-$.
Denote moreover by $\Phi = W(\Delta) = \bigset{w(\alpha_s)}{w \in W, s \in \sref} \subseteq L^+ \sqcup L^-$ the \Dfn{root system} for $(W,\sref)$, by $\Phiplus = \Phi \cap L^+$ the \Dfn{positive roots}, and by $\Phipm = \Phiplus \sqcup -\Delta \subseteq \Phi$ the \Dfn{almost positive roots}.
We often use the letter $N = |\Phiplus|$, so that $n+N = |\Phipm|$.

\medskip

Let~$\Q$ be a word in the simple system~$\sref$ and let $\rho \in W$.
The \Dfn{subword complex} $\subwordComplex(\Q,\rho)$ is the simplicial complex of (positions of) letters in~$\Q$ whose complement contains a reduced word of~$\rho$.
These complexes were introduced by A.~Knutson and E.~Miller in~\cite{KM2004}, we refer to \Cref{ex:subwordA2} on page~\pageref{ex:subwordA2} for a detailed example.
Observe that $\subwordComplex(\Q,\rho)$ is pure by construction with ground set $[m] = \{1,\ldots,m\}$ given by the indices of letters in~$\Q = \q_1\cdots\q_m$.
Its facets thus all have the same size and we consider them as sorted lists of integers, written in set notation.
This is, $I = \{i_1 < \ldots < i_n\}$ is a facet of $\subwordComplex(\Q,\rho)$ if and only if the word $\q_1\cdots \widehat \q_{i_1} \cdots \widehat \q_{i_n} \cdots \q_m$ with the letters $\q_{i_1},\ldots \q_{i_n}$ omitted is a reduced word for~$\rho \in W$.

\medskip

Recall the following fundamental observation about subword complexes.
It explains in all later constructions the independence of the chosen reduced word~$\c = \s_1\cdots\s_n$ of the Coxeter element~$c$.

\begin{lemma}[{\cite[Proposition~3.8]{CLS2014}}]\label{lem:commutation}
  Let~$\Q'$ be a word in~$\sref$ that coincides with~$\Q$ up to commutations of consecutive letters that commute in~$W$.
  Then $\subwordComplex(\Q,\rho) \cong \subwordComplex(\Q',\rho)$, and the isomorphism is given by the natural identification between letters in~$\Q$ and in~$\Q'$.
  \qed
\end{lemma}

In this paper, we are only interested in the case that~$\rho = \wo \in W$ is the unique longest element with respect to the weak order, and $\Q$ being one specific word constructed from the Coxeter element~$c$.
We thus write $\subwordComplex(\Q)$ for $\subwordComplex(\Q,\wo)$ and assume that~$\Q$ does indeed contain a reduced word for~$\wo$.
This immediately implies that $\subwordComplex(\Q)$ is a simplicial sphere, see~\cite[Theorem~3.7]{KM2004}.
Define~$\greedy$ and $\antigreedy$ to be the lexicographically first and last facets of $\subwordComplex(\Q)$, respectively.
These are called \Dfn{greedy facet} and \Dfn{antigreedy facet}.

\medskip

For $\Q = \q_1 \cdots \q_m$, associate to any facet~$I$ of the subword complex~$\subwordComplex(\Q)$ a \Dfn{root function} $\Root{I}{\cdot} : [m] \to W(\Delta)$ and a \Dfn{weight function}~$\Weight{I}{\cdot} : [m] \to W(\nabla)$ defined by
\[
  \Root{I}{k} = \wordprod{\Q}{[k-1] \ssm I}(\alpha_{q_k}) \quad \text{and} \quad \Weight{I}{k} = \wordprod{\Q}{[k-1] \ssm I}(\omega_{q_k}),
\]
where~$\wordprod{\Q}{X}$ denotes the product of the simple reflections~$q_x \in \Q$, for~$x \in X$, in the order given by~$\Q$.
For later convenience, we as well define the \Dfn{coroot function} $\coRoot{I}{\cdot} : [m] \to W(\Delta^\vee)$ and a \Dfn{coweight function}~$\coWeight{I}{\cdot} : [m] \to W(\nabla^\vee)$ by
\[
  \coRoot{I}{k} = \wordprod{\Q}{[k-1] \ssm I}(\alpha_{q_k}^\vee) \quad \text{and} \quad \coWeight{I}{k} = \wordprod{\Q}{[k-1] \ssm I}(\omega_{q_k}^\vee).
\]

Observe that it is immediate from this definition that all these functions are invariant under the isomorphism given in \Cref{lem:commutation}.
The root function (and, equivalently, the coroot function) locally encodes the flip property in the subword complex: each facet adjacent to~$I$ in~$\subwordComplex(\Q)$ is obtained by exchanging an element~${i \in I}$ with the unique element~$j \notin I$ such that $\Root{I}{j} \in \{ \pm \Root{I}{i} \}$.
If $i < j$ such a flip is called \Dfn{increasing}, and it is called \Dfn{decreasing} otherwise.
Observe that the greedy facet and the antigreedy facet are the unique facets such that every flip is increasing and decreasing, respectively.
After this exchange, the root function and the weight function are updated by an application of~$s_{\Root{I}{i}}$ as recalled in \Cref{lem:lemma33}.
The root and the coroot functions are used to define the \Dfn{root configuration} and the \Dfn{coroot configuration} of the facet~$I$ as the ordered multisets
\begin{align}
  \Roots{I} = \bigmultiset{\Root{I}{i}}{i \in I},\quad \coRoots{I} = \bigmultiset{\coRoot{I}{i}}{i \in I}. \label{eq:rootconfiguration}
\end{align}

Similarly, the weight and the coweight functions are used to define the \Dfn{weight configuration} and the \Dfn{coweight configuration}
\begin{align}
  \Weights{I} = \bigmultiset{\Weight{I}{i}}{i \in I},\quad \coWeights{I} = \bigmultiset{\coWeight{I}{i}}{i \in I}. \label{eq:weightconfiguration}
\end{align}

By ordered multiset we simply mean the ordered tuple written in set notation.
For later convenience, we denote by
\[
  \Root{I}{i}_j = \dotprod{\Root{I}{i}}{\omega_j^\vee}
\]
the coefficient of~$\alpha_j$ in the root $\Root{I}{i}$.

\medskip

On the other hand, the weight function is used to define the \Dfn{brick vector} of~$I$ as
\[
  \brickVector(I) = \sum_{k \in [m]} \Weight{I}{k},
\]
and the \Dfn{brick polytope} of~$\Q$ is defined to be the convex hull of the brick vectors of all facets of the subword complex~$\subwordComplex(\Q)$,
\[
  \brickPolytope(\Q) = \conv \bigset{\brickVector(I)}{I \text{ facet of } \subwordComplex(\Q)}.
\]
It is shown in~\cite{PS2015} that the brick polytope~$\brickPolytope(\Q)$ is the Minkowski sum of a Coxeter matroid polytope in the sense of~\cite{BGW2003}.

\begin{theorem}[{\cite[Proposition~1.5]{PS2015}}]\label{thm:minkowski}
  For any word $\Q$ in~$\sref$ of length~$m$ containing a reduced word for~$\wo$ we have that
  \[
    \brickPolytope(\Q) = \Mink_{k \in [m]} \brickPolytope(\Q,k)
  \]
  where $\brickPolytope(\Q,k) = \conv \bigset{\Weight{I}{k}}{I \text{ facet of } \subwordComplex(\Q)}$.
  \qed
\end{theorem}

For the Coxeter element~$c$ with fixed reduced word $\c = \s_1\cdots\s_n$, the \Dfn{Coxeter-sorting word} (or \Dfn{$\c$-sorting word}) $\c(\rho)$ of an element $\rho \in W$ is given by the lexicographically first subword of $\c^\infty$ that is a reduced word for~$\rho$.
Observe that the word $\c(\rho)$ does depend on the chosen reduced word, but that it should instead been thought of as being associated to the Coxeter element~$c$ and being defined up to commutations of consecutive commuting letter.
The notion of $c$-sorting words was defined by N.~Reading in~\cite{R2007} and plays a crucial role in the combinatorial descriptions of finite type cluster algebras and in particular in the description of cluster complexes in terms of subword complexes.
The main results in~\cite{CLS2014} provides the following description of the combinatorics of the cluster complex of $\Alg(W,c)$, where we observe from \Cref{lem:commutation} that the subword complex does not depend on the chosen reduced word for~$c$.

\begin{theorem}[{\cite[Theorem~2.2]{CLS2014}}]\label{thm:cluster1}
  The cluster complex of the cluster algebra $\Alg(W,c)$ is isomorphic to the subword complex $\clusterComplex$.
  \qed
\end{theorem}

We thus refer to $\clusterComplex$ as the \Dfn{$c$-cluster complex}.
We moreover remark that the abstract simplicial complex $\clusterComplex$ does not depend on the chosen Coxeter element~$c$, while its combinatorics in the sense of root and weight functions does depend on~$c$.
This is, $\subwordComplex\big(\cw{c}\big) \cong \subwordComplex\big(\cw{c'}\big)$ for any two Coxeter elements~$c$ and~$c'$ with reduced words~$\c$ and~$\c'$, respectively, see~\cite[Theorem~2.6]{CLS2014}.

\medskip

One identifies positions in $\cw{c}$ and almost positive roots by sending the $k$\th\ letter $\s_{k}$ ($1 \leq k \leq n$) of the initial copy of~$\c$ to the negative simple root $-\alpha_k$, and the $k$\th\ letter $\q_k$ ($1 \leq k \leq N$) of $\sw{c}{w_{\circ}}$ to the positive root $q_1\cdots q_{k-1}(\alpha_{q_k})$.
See \Cref{lem:lemma33}\eqref{it:bijectionpositiveroots} on page~\pageref{lem:lemma33} that this indeed is a bijection, and observe that this equals in the natural way the root function of the greedy facet and of the antigreedy facet.
In symbols, this is for $1 \leq k \leq N$
\begin{align}
  q_1\cdots q_{k-1}(\alpha_{q_k}) = \Root{\antigreedy}{k} = \Root{\greedy}{n+k}. \label{eq:rootbijection}
\end{align}

That identification yields the isomorphism in \Cref{thm:cluster1} by sending a cluster to the positions inside the word~$\cw{c}$ corresponding to the almost positive roots of the $d$-vectors of the cluster.
To make this explicit, we use the following notation:
let~$I = \{ i_1 < \ldots < i_n\}$ be a facet of the cluster complex $\clusterComplex$.
We then denote by $S(I) = \big(\mutmatrix(I),\u(I),\f(I)\big)$ with
\begin{align*}
  \mutmatrix(I) &= \begin{bmatrix} \Mpr(I) \\ \Mex(I) \end{bmatrix} \\
  \u(I) &= \big(u_{i_1}(I),\ldots,u_{i_n}(I)\big) \\
  \f(I) &= \big(f_{i_1}(I),\ldots,f_{i_n}(I)\big)
\end{align*}
the cluster seed of $\Alg(W,c)$ corresponding to~$I$ under the given isomorphism between cluster variables, almost positive roots, and positions in the word $\cw{c}$.
The columns of $\mutmatrix(I)$ are then also indexed by the positions $i_1,\ldots,i_n$ of~$I$ as are the rows of $\Mpr(I)$, while the rows of $\Mex(I)$ are indexed by the positions $1,\ldots,n$ (which are the positions of the greedy facet~$\greedy = \{1,\ldots,n\}$).
We also denote by~$\cvec{I,i}$ the $c$-vector coming from column $i \in I$ of $\Mex(I)$, and by $\gvec{I,i}$ the $g$-vector of the entry $u_i(I)$.

\medskip

Polar polytopal realizations of the cluster complex were first obtained by F.~Cha\-po\-ton, S.~Fomin, and A.~Zelevinsky in~\cite{CFZ2002}.
C.~Hohlweg, C.~Lange, and H.~Thomas then constructed in~\cite{HLT2011} a generalization depending on a Coxeter element~$c$, that reduces for bipartite~$c$ to the construction in~\cite{CFZ2002}.
As one obtains for type~$A_n$ classical constructions of associahedra, such polytopal realizations are called \Dfn{$c$-associahedra}.
The subword complex approach and the brick polytope construction provide a rather simple construction of these.

\begin{theorem}[{\cite[Theorem~4.9]{PS2015}}]\label{thm:cluster2}
  The cluster complex of $\Alg(W,c)$ is realized by the polar of the brick polytope~$\brickPolytope\big(\cw{c}\big)$.
  \qed
\end{theorem}

This polytopal realization turns out to be equal to the construction in~\cite{HLT2011} up to a translation, see~\cite[Corollary~6.10]{PS2015}.
Its main advantage is that is provides a vertex description that yields a very simple proof of \Cref{thm:cluster2}.
This construction lives inside the weight space, while its natural translation by $\brickVector(\antigreedy)$ lives inside the root space.
We will conjecturally see in \Cref{conj:newtonpolytope} that this is closely related to the $F$-polynomials also ``living inside the root space''.

\medskip

Next, we recall how the indexing of the principal part of the exchange matrix is chosen, and why one can think of it as a \emph{matrix of scalars}.

\begin{theorem}[{\cite[Theorem~6.40]{PS2015}}]\label{thm:exchangematrix}
  Let~$I$ be a facet of $\clusterComplex$.
  The principal part of the exchange matrix $\mutmatrix(I)$ is then given for $i,j \in I$ by
  \begin{align*}
    [\Mpr(I)]_{ij} =
    \begin{cases}
        -\bigdotprod{\Root{I}{j}}{\coRoot{I}{i}} & \text{ if } i < j, \\[5pt]
        \phantom{-}\bigdotprod{\Root{I}{j}}{\coRoot{I}{i}} & \text{ if } i > j, \\[5pt]
        \phantom{-}0 & \text{ if } i = j.
    \end{cases}
  \end{align*}
  \qed
\end{theorem}

Observe that one could directly express this as well in terms of the skew-sym\-metric bilinear form defined by \Cref{eq:bmatrix}.

\medskip

The following remark starts to clarify the connection between the subword complex approach to finite type cluster algebras and the approach using N.~Reading and D.~Speyer's combinatorial frameworks~\cite{RS2016}.

\begin{remark}\label{rem:RS1}
  The central structures in their combinatorial frameworks are the \emph{labels} and \emph{colabels}.
  It follows from~\cite[Proposition~6.20]{PS2015} that the labels in finite types are the root configurations defined in~\eqref{eq:rootconfiguration}, and we obtain by duality that the colabels in finite types are the coroot configurations also defined in~\eqref{eq:rootconfiguration}.
  Given this connection in finite types, we immediately obtain that \Cref{thm:exchangematrix} is the same description of the principal part of the exchange matrix as given in~\cite[Theorem~3.25]{RS2016}.
  See \Cref{rem:RS2} below for the relation of the subword complex approach and~\cite[Theorem~3.26]{RS2016}.
\end{remark}

Before presenting the results of this paper, we explain them in great detail the example of type $A_2$ for the reader's convenience.

\begin{example}\label{ex:subwordA2}
  This example shows the root and the weight function of type~$A_2$, together with the construction of the~$c$-cluster complex $\clusterComplex$.
  It is presented in order to emphasize the \emph{close similarity} to the type~$A_2$ cluster algebra in \Cref{ex:clusterA2}.
  Let~$W$ be the symmetric group $A_2 = \mathfrak{S}_3$ with simple transpositions
  \[
    \sref=\big\{\tau_1 = (12),\tau_2 = (23)\big\},
  \]
  Coxeter element $c = \tau_1\tau_2 = (123)$, simple roots
  \[
    \Delta = \{ \alpha_1 = (1,-1,0), \alpha_2 = (0,1,-1) \},
  \]
  and fundamental weights
  \[
    \nabla = \{ \omega_1 = (1,0,0), \omega_2 = (1,1,0) \}.
  \]
  As the space is given by $\RR^3 \big/ \RR(1,1,1)$, we have that
  \[
    \alpha_1 = 2\omega_1-\omega_2, \quad \alpha_2 = 2\omega_2-\omega_1,
  \]
  compare~\eqref{eq:rootweight}.
  The word $\Q=\cw{c}$ is given by
  \[
    \q_1\q_2\ \q_3\q_4\q_5=\underbrace{\tau_1\tau_2}_{\c}\ \underbrace{\tau_1\tau_2\tau_1}_{\sw{w_{\circ}}{c}},
  \]
  and the facets of $\clusterComplex$ are
  \[
    \{1,2\},\{2,3\},\{3,4\},\{4,5\},\{1,5\}.
  \]
  The following table records the root function of $\clusterComplex$ indexed both by almost positive roots and positions in the word $\Q$:
  \[
    \begin{array}{|rc|C{40pt}|C{40pt}|C{40pt}|C{40pt}|C{40pt}|}\hline
          && -\alpha_1 & -\alpha_2 & \alpha_1 & \alpha_1+\alpha_2 & \alpha_2 \\
     & I & 1 & 2 & 3 & 4 & 5 \\\hline\hline

      \multirow{2}{*}{$\greedy =$\hspace*{-5pt}} & \multirow{2}{*}{$\{1,2\}$} & \lightgrey \alpha_1 & \lightgrey \alpha_2 & \alpha_1 & \alpha_1+\alpha_2 & \alpha_2 \\
      && \lightgrey (1,-1,0) & \lightgrey (0,1,-1) & (1,-1,0) & (1,0,-1) & (0,1,-1) \\\hline
      &\multirow{2}{*}{$\{2,3\}$}  & \alpha_1 & \lightgrey \alpha_1+\alpha_2 & \lightgrey -\alpha_1 & \alpha_1+\alpha_2 & \alpha_2 \\
      && (1,-1,0) & \lightgrey (1,0,-1) & \lightgrey (-1,1,0) & (1,0,-1) & (0,1,-1) \\\hline
      &\multirow{2}{*}{$\{3,4\}$}  & \alpha_1 & \alpha_1+\alpha_2 & \lightgrey \alpha_2 & \lightgrey -\alpha_1-\alpha_2 & \alpha_2 \\
      && (1,-1,0) & (1,0,-1) & \lightgrey (0,1,-1) & \lightgrey (-1,0,1) & (0,1,-1) \\\hline
      \multirow{2}{*}{$\antigreedy =$}\hspace*{-5pt} &\multirow{ 2}{*}{$\{4,5\}$} & \alpha_1 & \alpha_1+\alpha_2 & \alpha_2 & \lightgrey -\alpha_1 & \lightgrey -\alpha_2 \\
      && (1,-1,0) & (1,0,-1) & (0,1,-1) & \lightgrey (-1,1,0) & \lightgrey (0,-1,1) \\\hline
      &\multirow{2}{*}{$\{1,5\}$}  & \lightgrey \alpha_1 & \alpha_2 & \alpha_1+\alpha_2 & \alpha_1 & \lightgrey -\alpha_2 \\
      && \lightgrey (1,-1,0) & (0,1,-1) & (1,0,-1) & (1,-1,0) & \lightgrey (0,-1,1) \\\hline
    \end{array}
  \]
  Observe that the root configuration of a facet~$I$ (indicated in grey) written in simple roots coincides with the $c$-vectors of the corresponding cluster seed in \Cref{ex:clusterA2}.
  E.g., the facet~$\{3,4\}$ corresponds to the cluster seed where the $d$-vectors are the almost positive roots $\big(\alpha_1,\alpha_1+\alpha_2\big)$.
  It has root configuration $\big(\alpha_2,-\alpha_1-\alpha_2\big)$ which corresponds to the two $c$-vectors $\cvec{\alpha_1} = (0,1) = \alpha_2$ and $\cvec{\alpha_1+\alpha_2} = (-1,-1) = -\alpha_1-\alpha_2$.
  This phenomenon will be explained in all finite types in \Cref{thm:extendedpart}.

  \medskip

  Similarly, the following table records the weight function of $\clusterComplex$:
  \[
    \begin{array}{|rc|C{33pt}|C{33pt}|C{33pt}|C{33pt}|C{33pt}||c|}\hline
          && -\alpha_1 & -\alpha_2 & \alpha_1 & \alpha_1+\alpha_2 & \alpha_2 & \\
      & I & 1 & 2 & 3 & 4 & 5 & \brickVector(I) \\\hline\hline

      \multirow{2}{*}{$\greedy =$\hspace*{-5pt}} & \multirow{2}{*}{$\{1,2\}$} & \lightgrey \omega_1 & \lightgrey \omega_2 & \omega_1 & \omega_2 & \omega_2-\omega_1 & \omega_1+3\omega_2\\
      & & \lightgrey (1,0,0) & \lightgrey (1,1,0) & (1,0,0) & (1,1,0) & (0,1,0) & (4,3,0) \\
      \hline

      &\multirow{2}{*}{$\{2,3\}$} & \omega_1 & \lightgrey \omega_2 & \lightgrey \omega_2-\omega_1 & \omega_2 & \omega_2-\omega_1 & -\omega_1+4\omega_2 \\
      & & (1,0,0) & \lightgrey (1,1,0) & \lightgrey (0,1,0) & (1,1,0) & (0,1,0) & (3,4,0) \\
      \hline

      &\multirow{2}{*}{$\{3,4\}$} & \omega_1 & \omega_2 & \lightgrey \omega_2-\omega_1 & \lightgrey -\omega_1 & \omega_2-\omega_1 & -2\omega_1+3\omega_2\\
      & & (1,0,0) & (1,1,0) & \lightgrey (0,1,0) & \lightgrey (0,1,1) & (0,1,0) & (2,4,1) \\
      \hline

      \multirow{2}{*}{$\antigreedy =$}\hspace*{-5pt} &\multirow{ 2}{*}{$\{4,5\}$} & \omega_1 & \omega_2 & \omega_2-\omega_1 & \lightgrey -\omega_1 & \lightgrey -\omega_2 & -\omega_1+\omega_2 \\
      & & (1,0,0) & (1,1,0) & (0,1,0) & \lightgrey (0,1,1) & \lightgrey (0,0,1) & (2,3,2) \\
      \hline

      &\multirow{2}{*}{$\{1,5\}$} & \lightgrey \omega_1 & \omega_2 & \omega_1 & \omega_1-\omega_2 & \lightgrey -\omega_2 & 3\omega_1-\omega_2\\
      & & \lightgrey (1,0,0) & (1,1,0) & (1,0,0) & (1,0,1) & \lightgrey (0,0,1) & (4,1,2) \\
      \hline
    \end{array}
  \]
  This yields that the brick polytope is given by
  \begin{align*}
    \brickPolytope\big(\cw{c}\big) =& \conv\big\{ 430, 340, 241, 232, 412 \big\} \\
                           =& \conv\{\omega_1\} + \conv\{\omega_2\} + \conv\{\omega_1,\omega_2-\omega_1\} \\
                            &+ \conv\{\omega_2,-\omega_1,\omega_1-\omega_2\} + \conv\{\omega_2-\omega_1,-\omega_2\}.
  \end{align*}

  There are multiple things to be observed in this table which will be conjectured/explained in this paper.
  First, the weight configuration~$I$ (again indicated in grey) written in fundamental weights coincides with the $g$-vectors of the corresponding cluster seed in \Cref{ex:clusterA2}.
  E.g., the facet $\{3,4\}$ has weight configuration $(\omega_2-\omega_1,-\omega_1)$ which corresponds to the two $g$-vectors $\gvec{\alpha_1} = (-1,1) = \omega_2-\omega_1$ and $\gvec{\alpha_1+\alpha_2} = (-1,0) = -\omega_1$.
  This phenomenon will be explained in all finite types in \Cref{cor:extendedpart}.
  Moreover, the weights inside a column are all equal within the entries inside the facets (the entries in grey) and these weights also coincide with the weight in the row of the antigreedy facet.
  This will be explained in \Cref{lem:weightequality} and the following paragraph.

  Next, and most importantly, one shifts all weights inside a column by the weight in the row of the antigreedy facet~$\antigreedy$ and expresses the result in terms of the simple roots to obtain in each column the exponent vectors of the monomials in the $F$-polynomials for the corresponding cluster variable:
  \[
    \begin{array}{|rc|c|c|c|c|c||c|}\hline
          && -\alpha_1 & -\alpha_2 & \alpha_1 & \alpha_1+\alpha_2 & \alpha_2 & \brickVector(I) - \brickVector(\antigreedy) \\\hline\hline

      \greedy =\hspace*{-5pt} &\{1,2\} & \lightgrey 0 & \lightgrey 0 & \alpha_1 & \alpha_1+\alpha_2 & \alpha_2 & 2\alpha_1+2\alpha_2 \\\hline
      &\{2,3\} & 0 & \lightgrey 0 & \lightgrey 0 & \alpha_1+\alpha_2 & \alpha_2 & \alpha_1+2\alpha_2 \\\hline
      &\{3,4\} & 0 & 0 & \lightgrey 0 & \lightgrey 0 & \alpha_2 & \alpha_2 \\\hline
      \antigreedy =\hspace*{-5pt} &\{4,5\} & 0 & 0 & 0 & \lightgrey 0 & \lightgrey 0 & 0
      \\\hline
      &\{1,5\} & \lightgrey 0 & 0 & \alpha_1 & \alpha_1 & \lightgrey 0 & 2\alpha_1 \\\hline\hline
      &              & 1 & 1 & 1    & 1       & 1   \\
      & \Fpoly{\beta}{\y}  &   &   & y_1  & y_1     & y_2 \\
      &              &   &   &      &  y_1y_2 &     \\
      \cline{1-7}
    \end{array}
  \]
  We will prove this phenomenon in type~$A_n$, while we will only conjecture generalizations thereof in general finite types which we verify in low ranks including all exceptional types.

  Nevertheless, the following properties of the columns perfectly match known properties of $F$-polynomials and hold for general finite type~$c$-cluster complexes:
  \begin{enumerate}[(i)]
    \item When shifting all weights inside the columns by the entries of the antigreedy facet~$\antigreedy$, all entries inside the facets become~$0$ and the row of the greedy facet~$\greedy$ coincides with the row of~$\greedy$ for the table of the root function in the positions corresponding to the positive roots (while the simple negative roots become~$0$ in this table), see \Cref{lem:rootlemma}.
    \item Every other entry is obtained from the entry of the greedy facet $\greedy$ (the antigreedy facet $\antigreedy$) by subtracting (adding) simple roots, see \Cref{lem:positiveweightshift}.
  \end{enumerate}
  The first item corresponds to the facts that $F$-polynomials have constant term~$1$ and a monomial with exponent vector equal to the $d$-vector, and the second item corresponds to the fact that this monomial is the unique monomial of highest degree and is divided by every other monomial in the $F$-polynomial.
\end{example}

\medskip

The first result of this paper shows the close relationship between the $c$-vectors in finite type cluster algebras and the root function of the corresponding subword complex.

\begin{theorem}\label{thm:extendedpart}
  Let~$I$ be a facet of the~$c$-cluster complex $\clusterComplex$ corresponding to the seed~$S(I)$ in the cluster algebra~$\Alg(W,c)$.
  Then the columns of $\Mex(I)$ are given by the root configuration, \ie,
  \[
    \cvec{I,i} = \Root{I}{i}
  \]
  for all $i \in I$.
  In particular, $\cvec{I,i} \in \Phi$ is sign-coherent and $\set{\cvec{I,i}}{i \in I}$ forms a lattice basis of the root lattice.
\end{theorem}

To emphasize the similarity of this result with \Cref{thm:exchangematrix}, we rewrite this result in terms of the frozen variables and the extended part of the mutation matrix as a \emph{matrix of scalars} and obtain $\f_i(I) = y_1^{\Root{I}{i}_1} \cdots y_n^{\Root{I}{i}_n}$, which means that the extended part of the exchange matrix $\mutmatrix(I)$ is given for $i \in \greedy$ and $j \in I$ by
\begin{align*}
  [\Mex(I)]_{ij} = \bigdotprod{\Root{I}{j}}{\coWeight{\greedy}{i}}.
\end{align*}
This explains why we think of the coluns of $\Mex(I)$ as being indexed by the almost positive roots in a facet, while we think of the rows as being indexed by the simple roots.

\begin{corollary}\label{cor:extendedpart}
  In the situation of \Cref{thm:extendedpart}, we also obtain that the $g$-vectors are given by the weight configuration,
  \[
    \gvec{I,i} = \Weight{I}{i}
  \]
  for all~$i \in I$.
  In particular, $\set{\gvec{I,i}}{i \in I}$ forms a lattice basis of the weight lattice.
\end{corollary}
\begin{proof}
  Given a cluster algebra $\Alg(W,c)$ as considered in \Cref{thm:extendedpart}, it was proven in \cite[Theorem 1.2]{NZ2011} that the $g$-matrix, whose columns consist of the $g$-vectors of $\Alg(W,c)$, is equal to the transpose inverse of the $c$-matrix, whose columns consist of the $c$-vectors of $\Alg(W,c)$.
  
  This fact along with \Cref{thm:extendedpart} tells us that the $g$-vectors form the dual basis to the coroot configuration, see also~\cite[Section~3.1]{RS2016}.
  The statement thus follows with the stronger version of \Cref{lem:lemma33}\eqref{it:weightroot} below which was established in~\cite[Proposition~6.6]{PS2015}.
\end{proof}

\begin{remark}\label{rem:RS2}
  We have seen in \Cref{rem:RS1} how the description of the mutation matrix through subword complexes relates to the description through N.~Reading and D.~Speyer's combinatorial frameworks.
  Indeed, \Cref{thm:extendedpart} is the subword complex counterpart of~\cite[Theorem~3.26]{RS2016}.
  \Cref{thm:extendedpart} is then the same as Theorem~3.26(1) and~(2), and \Cref{cor:extendedpart} implies Theorem~3.26(3) and~(4).
  We also remark that~\cite[Theorem~5.39]{RS2016} provides a way of computing the $g$-vectors in finite types using their combinatorial framework.
  Theorem~3.26(5) then states that all $F$-polynomials in finite type have constant term~$1$, this follows by the same argument via~\cite[Proposition~5.6]{FZ2007}.
\end{remark}

We have now seen how to obtain properties from the root and coroot configurations, and also from the weight configuration.
Indeed, we have \emph{not} used the root and coroot functions outside of the facets to derive information.
This does not seem very surprising in light of \Cref{lem:lemma33}\eqref{it:bijectionpositiveroots} which recalls that the root function on the complement of a given facet is always the complete set of positive roots.

\medskip

Next, we \emph{look at properties of the cluster algebra that can be studied using the weight function}, this time outside of a given facet.
Indeed, we will see that we can conjecturally obtain further desired information about the cluster algebra from this weight function, see \Cref{conj:newtonpolytope}, \Cref{cor:dvector,cor:gvector}, and \Cref{thm:newton} below.

\medskip

To state the main conjecture of this paper, we define the \Dfn{Newton polytope} of an $F$-polynomial $\Fpoly{\beta}{\y}$ as the convex hull of its exponent vectors in the root basis~$\Delta$.
That is,
\[
  \Newton{\Fpoly{\beta}{\y}} = \conv\bigset{\lambda_1 \alpha_1+\ldots+\lambda_n\alpha_n}{y_1^{\lambda_1}\cdots y_n^{\lambda_n}\text{ monomial in }\Fpoly{\beta}{\y}}.
\]

\begin{conjecture}\label{conj:newtonpolytope}
  Let $\Fpoly{\beta}{\y}$ be the $F$-polynomial associated to the positive root~$\beta$ for the cluster algebra $\Alg(W,c)$.
  Let~$k$ be the unique index $k \in \{n+1,\ldots,n+N\}$ such that $\Root{\greedy}{k} = \beta$ associated to~$\beta$ in Equation~\eqref{eq:rootbijection}.
  Then
  \[
    \Newton{\Fpoly{\beta}{\y}} = \conv\set{\Weight{I}{k} - \Weight{\antigreedy}{k}}{I \text{ facet of }\clusterComplex}.
  \]
\end{conjecture}

We moreover conjecture that knowing the Newton polytope of an $F$-polynomial $\Fpoly{\beta}{\y}$ in a finite type cluster algebra is enough to recover all monomials in $\Fpoly{\beta}{\y}$.

\begin{conjecture}\label{conj:latticepoints}
  The exponent vectors of the monomials of the $F$-polynomial $\Fpoly{\beta}{\y}$ are given by
  all lattice points inside its Newton polytope,
  \[
    \bigset{\lambda_1 \alpha_1+\ldots+\lambda_n\alpha_n}{y_1^{\lambda_1}\cdots y_n^{\lambda_n}\text{ monomial in }\Fpoly{\beta}{\y}} = \Newton{\Fpoly{\beta}{\y}} \cap L^+.
  \]
\end{conjecture}

We emphasize that we do currently not have an explicit conjecture what the coefficients of the monomials look like---it is planned to investigate this in future research.

\medskip

\begin{theorem}\label{thm:main1}
  \Cref{conj:newtonpolytope,conj:latticepoints} hold for~$\Alg(W,c)$ with~$W$ of type~$A_n$.
\end{theorem}
This theorem will be proved in \Cref{sec:proofmain1} by relating it to the combinatorial model of type~$A_n$ cluster algebras of R.~Schiffler~\cite{S2008} using its description given by G.~Musiker and R.~Schiffler in~\cite{MS2010}.

\begin{remark}
  The combinatorial model for $F$-polynomials for cluster algebras from punctured surfaces can as well be used to provide the $F$-polynomials for type~$D$ cluster complexes via the description given for example by C.~Ceballos and V.~Pilaud in~\cite{CP2015}.
  We expect that constructions similar to those we use in \Cref{sec:proofmain1} to derive \Cref{thm:main1} can as well be given to prove \Cref{conj:newtonpolytope} in type~$D$.
  We plan to also further investigate this explicit combinatorial approach.
\end{remark}

\begin{theorem}\label{thm:main2}
  \Cref{conj:newtonpolytope,conj:latticepoints} hold for~$\Alg(W,c)$ with~$W$ being of rank at most~$8$.
  In particular, they hold in all exceptional types.
\end{theorem}
\begin{proof}
  This was obtained via explicit computer explorations\footnote{The computations were performed using \texttt{sage-7.2}. The development was supported by the project ``Combinatorial and geometric structures for reflection groups and groupoids'' within the German Research Foundation priority program ``Algorithmic and experimental methods in Algebra, Geometry, and Number Theory''. A worksheet with the code and examples is available upon request.}.
\end{proof}

Whenever the first of these conjectures holds, we obtain very simple combinatorial proofs of many properties of finite type cluster algebras that were already conjectured in~\cite{FZ2007}.
Since then, all those properties were proven in acyclic finite types, but often using rather intricate machinery, while the proofs here are elementary once the needed combinatorial properties of subword complexes are established.
We refer to~\cite[Section~3.3]{RS2016} and in particular the table at the end of that section for references to proofs of these properties.

\medskip

The first two corollaries describe how to obtain properties of the $F$-polynomial from \Cref{conj:newtonpolytope}, and we then describe what is missing to obtain the actual cluster variables from the root and the weight function.

\begin{corollary}[{\cite[Conjecture~7.17]{FZ2007}}]\label{cor:dvector}
  Assume that \Cref{conj:newtonpolytope} holds for $\Alg(W,c)$.
  For any $\beta \in \Phiplus$, we then have that the $F$-polynomial $\Fpoly{\beta}{\y}$ has a unique monomial of maximal degree whose exponent vector equals~$\beta$, and such that any of its monomials divides this monomial of maximal degree.
\end{corollary}
\begin{proof}
  This directly follows from \Cref{lem:positiveweightshift,lem:rootlemma} below.
  Observe that we do not need the information about all monomials in the $F$-polynomials to deduce the statement, but that the information about the monomials corresponding to the vertices of the Newton polytope is indeed sufficient.
\end{proof}

Using~\cite[Proposition~7.16]{FZ2007}, we can now also deduce how to compute the $g$-vector from the $F$-polynomial, and indeed from the weight function alone.

\begin{corollary}[{\cite[Conjecture~6.11]{FZ2007}}]\label{cor:gvector}
  Assume that \Cref{conj:newtonpolytope} holds for $\Alg(W,c)$.
  For any $\beta \in \Phiplus$, we then have that the $g$-vector $\gvec{\beta}$ is given by
  \[
    \gvec{\beta} = \max\big(\Fpoly{\beta}{\hat\x}\big) - \beta
  \]
  where $\hat\x = (\hat x_1,\ldots, \hat x_n)$ with $\hat x_i = \prod_j x_j^{-[\sfM_c]_{ji}}$, and $\max(\Fpoly{\beta}{\hat\x})$ denotes the componenentwise maximum of the exponent vectors of the monomials in $\Fpoly{\beta}{\hat\x}$.
  Moreover, this maximum is obtained when only considering exponent vectors of monomials that correspond to vertices of $\Newton{\Fpoly{\beta}{\y}}$.
\end{corollary}

Observe that one usually uses tropical notation to express $\max\big(\Fpoly{\beta}{\hat\x}\big)$ as briefly used in the second paragraph of \Cref{sec:proofextendedpart} below.
We chose to show it in the present form for simplicity in the given context.

\begin{proof}[Proof of \Cref{cor:gvector}]
  The equality follows from \Cref{cor:dvector} via~\cite[Proposition~7.16]{FZ2007}.
  The following simple fact about polytopes implies that it is enough to consider vertices of the Newton polytope.
  Let~$P \in \RR^n$ be a polytope, and let~$\hat P$ is the image of~$P$ under a linear map $\varphi:\RR^n\rightarrow\RR$.
  Then the maximum of~$\hat P$ is obtained at a vertex of~$P$.
  This is,
  \[
    \max{\hat P} = \max\set{\varphi(v)}{v \text{ vertex of }P}.
  \]
  Applying this observation to every component yields the desired restriction to vertices of the Newton polytope.
\end{proof}

To state the main implication of the conjecture, we recall in the following proposition how to recover the cluster variable from the the $g$-vector and the $F$-polynomial.

\begin{proposition}[{\cite[Corollary~6.3]{FZ2007}}]\label{prop:clustervariables}
  For any $\beta \in \Phiplus$, we have that the cluster variable $\u_\beta(\x,\y)$ is given by
  \[
    \u_\beta(\x,\y) = x_1^{g_1}\cdots x_n^{g_n} \Fpoly{\beta}{\hat\y},
  \]
  where we use the $g$-vector $\gvec{\beta} = (g_1,\ldots,g_n)$ and 
  where we set $\hat\y = (\hat y_1,\ldots, \hat y_n)$ with $\hat y_i = y_i \hat x_i^{-1} = y_i \cdot \prod_j x_j^{[\sfM_c]_{ji}}$.
  \qed
\end{proposition}

\begin{theorem}\label{thm:newton}
  Assume that \Cref{conj:newtonpolytope} holds for $\Alg(W,c)$.
  Then the~$c$-associahedron $\brickPolytope\big(\cw{c}\big)$ coincides up to translation with the two Minkowski sums
  \begin{itemize}
    \item $\sum_{\beta\in\Phiplus}\Newton{\Fpoly{\beta}{\y}} \subset V$, or
    \vspace*{5pt}
    \item $\sum_{\beta\in\Phiplus}\Newton{u_\beta(\x,\y)} \subset \varphi(V) \subset V \oplus V$
  \end{itemize}
  for a suitable affine embedding~$\varphi$ of~$V$ into~$V \oplus V$.
\end{theorem}
\begin{proof}
  The first description in terms of the $F$-polynomials follows from the Min\-kowski decomposition of any brick polytope into Coxeter matroid polytopes recalled in \Cref{thm:minkowski}.

  The second description then in terms of the cluster variables follows from the description in terms of the $F$-polynomials using \Cref{prop:clustervariables} as this shows that the cluster variables depend affinely on the $\y$-variables, so that
  \[
    \sum_\beta \Newton{u_\beta(\x,\y)} \subseteq V \oplus V
  \]
  indeed lives inside the affine embedding~$\varphi(V) \subseteq V \oplus V$ given in \Cref{prop:clustervariables}.
\end{proof}

\begin{remark}\label{rem:postnikov1}
  For the \emph{linear Coxeter element} $c = (1,\ldots,n+1)$ in type~$A_n$, this description is equivalent to the description given by A.~Postnikov in~\cite{Pos2009}, see Corollary~8.2 and the following two paragraphs\footnote{We thank Vincent Pilaud for pointing out this connection.}.
  There, it is shown in this case that the Minkowski sum of the Newton polytopes of the $F$-polynomials is exactly the realization given by J.-L.~Loday in~\cite{Lod2004}.
\end{remark}

\Cref{thm:newton} immediately suggests the following conjecture.
Let~$\Alg(\mutmatrix)$ be a finite type cluster algebra with principal coefficients and possibly cyclic initial seed~$(\mutmatrix,\x,\y)$.
Define the \Dfn{$\mutmatrix$-associahedron} to be the Minkowski sum of the Newton polytopes of the $F$-polynomials of the cluster algebra.
This is the sum
\[
  \sum \Newton{\FpolyGen{\y}} \subset V,
\]
ranging over all $F$-polynomial $\FpolyGen{\y}$ of $\Alg(\mutmatrix)$.

\begin{conjecture}
  All $\mutmatrix$-associahedra of a given finite type have the same combinatorial type, \ie, they all have the same face lattice.
  In particular, the duals of the $\mutmatrix$-associahedra realize the cluster complex.
\end{conjecture}

In light of \Cref{prop:clustervariables}, one could use instead cluster variables in this definition for cluster algebras with principal coefficients.
Indeed, this conjecture would even make sense (in a slightly weaker form) in infinite types.
Taking the Minkowski sum corresponds to taking the finest common coarsening of the normal fans; one could thus also consider the finest (infinite) coarsening of the normal fans of the Newton polytopes of the $F$-polynomial in infinite types.

\begin{remark}
  The description of the $\mutmatrix$-associahedron using the finest coarsening of the normal fans of the Newton polytopes of the cluster variables was already conjectured by D.~Speyer and L.~Williams in~\cite[Conjecture~8.1]{SW2005} via the language of tropical geometry.
  They consider the variety of a cluster algebra $\Alg(\mutmatrix)$ of finite type and the positive part of its tropicalization and conjecture that whenever the finite type mutation matrix $\mutmatrix$ has full rank (which is the case for principal coefficients), the common refinement of the normal fans of the Newton polytopes of the cluster variables should be the fan given by the cluster complex of $\Alg(\mutmatrix)$ as the $d$-vector fan.
  Thus, \Cref{thm:newton} would imply their conjecture in the case of acyclic finite type cluster algebras with principal coefficients.

  The conjecture further states that if the mutation matrix does not have full rank, then the common refinement of the normal fans of the Newton polytopes of the cluster variables should be a coarsening of the fan dual to the cluster complex of $\Alg(\mutmatrix)$.
  C.~Ceballos, J.-P.~Labb\'e, and the first author proved this conjecture in type~$D_4$ in~\cite{BCL2015}.
\end{remark}

\section{Proof of \Cref{thm:extendedpart}}\label{sec:proofextendedpart}

In this section, we prove \Cref{thm:extendedpart} and also provide several auxiliary results for general finite type~$c$-cluster complexes, which will be used in \Cref{sec:proofmain1} to show the close relationship between $F$-polynomials and weight functions in type~$A_n$.

\medskip

We start with recalling cluster mutations on cluster seeds.
Let $S=\big(\mutmatrix, \u, \f \big)$ with  $\mutmatrix = \begin{bmatrix} \Mpr \\ \Mex \end{bmatrix}$ be a cluster seed as above.
Given that we have indexed columns of $\mutmatrix$ and the rows of $\Mpr$ both by the $d$-vectors of the cluster variables $\u = (u_1,\ldots,u_n)$, we now mutate~$S$ at $\beta \in \Phipm$ such that $\beta = \dvec{u_i}$.
The \Dfn{seed mutation} $\mu_i = \mu_\beta$ in \emph{direction}~$\beta$ defines a new seed $\mu_i(S)=(\mutmatrix', \x', \y')$ defined by the following \Dfn{exchange relations}, written for better readability in the indices $\{1,\ldots,n\}$ of $\{u_1,\ldots,u_n\}$ rather than in their $d$-vectors:

\begin{itemize}

  \item The entries of $\mutmatrix'=(b_{k\ell}')$ are given by
  \begin{align}
    b_{k\ell}'=
    \begin{cases}
      -b_{k\ell}             & \mbox{if } k=i \mbox{ or } \ell=i \\
      b_{k\ell}+b_{ki}b_{i\ell} & \mbox{if } b_{ki}>0 \mbox { and } b_{i\ell}>0 \\
      b_{k\ell}-b_{ki}b_{i\ell} & \mbox{if } b_{ki}<0 \mbox{ and } b_{i\ell}<0 \\
      b_{k\ell}              & \mbox{otherwise}.
    \end{cases}\label{it:matrixmutation}
  \end{align}

  \item The cluster variables $u_k'$ of the cluster $\u'=\{u_1',\ldots,u_n'\}$ are given by $u_k' = u_k$ for $k \neq i$ and
  \begin{align}
    u_i'=\frac{f_i\prod u_k^{\max\{b_{ki},0\}}+\prod u_k^{\max \{-b_{ki},0\} }}{(f_i\oplus 1)u_i}\label{it:clustermutation}
  \end{align}

  \item The frozen variables $f_\ell'$ of the coefficients $\f=\{f_1',\ldots,f_n'\}$ are given by 
  \begin{align}
    f_\ell'= \begin{cases}
            f_i^{-1} & \mbox{if } \ell=i \\
            f_\ell f_i^{\max \{b_{i\ell},0\}}(f_i\oplus 1)^{-b_{i\ell}} & \mbox{if } \ell \neq i
          \end{cases}.\label{it:coefficientmutation}
  \end{align}
\end{itemize}

As usual, we use the tropical notation $\oplus$ in \Cref{it:clustermutation,it:coefficientmutation}, which is defined for monomials by $\left(\prod_i y_i^{a_i}\right) \oplus \left(\prod_i y_i^{b_i}\right) = \prod_i y_i^{\min\{a_i,b_i\}}$.
It is worthwhile to compare this with the notation used in \Cref{cor:gvector}.

\medskip

A direct consequence of the definition is the following description of the frozen variables.

\begin{lemma}[{\cite[Eq.~(2.13)]{FZ2007}}]
  The frozen variables are given by
  \[
    f_i = y_1^{[\Mex]_{1i}} \cdots y_n^{[\Mex]_{ni}}.
  \]
  \qed
\end{lemma}

To prove \Cref{thm:extendedpart}, we will show that the entries in the root configuration behave as the~$c$-vectors described in the matrix mutation in \Cref{it:matrixmutation}.
In order to properly set this up, it is convenient to extract the coefficient of~$\alpha_j$ in the root $\Root{I}{i}$ using the inner product with the fundamental coweights, so that we aim to show that
\begin{align}
  \big[\Mex(I)\big]_{ji} = \Root{I}{i}_j = \dotprod{\Root{I}{i}}{\omega_j^\vee}. \label{eq:root_coefficient}
\end{align}

\medskip

The argument follows the same lines as the proof of \Cref{thm:exchangematrix} in~\cite{PS2015}.
We frequently make use of the following properties of the root and the weight function.

\begin{lemma}[{\cite[Lemma~3.3 \& Lemma~3.6]{CLS2014}, \cite[Lemma~3.3, Lemma~4.4 \& Proposition~6.6]{PS2015}}]\label{lem:lemma33}
  Let~$I$ and~$J$ be two adjacent facets of the subword complex $\subwordComplex(\Q)$ with $I\setminus i = J \setminus j$.
  Then
  \begin{enumerate}

    \item\label{it:bijectionpositiveroots} The map $\Root{I}{\cdot}: k \mapsto \Root{I}{k}$ is a bijection between the complement of~$I$ and~$\Phi^+$.

    \item\label{it:rootflip} The position~$j$ is the unique position in the complement of~$I$ for which $\Root{I}{j}\in \{\pm \Root{I}{i}\}$.
    Moreover, $\Root{I}{j}=\Root{I}{i}\in \Phi^+$ if $i<j$, while $\Root{I}{j}=-\Root{I}{i}\in \Phi^-$ if $j<i$.

    \item\label{it:rootupdate} The map $\Root{J}{\cdot}$ is obtained from $\Root{I}{\cdot}$ by
    \[
      \Root{J}{k}=\left\{\begin{array}{cc}s_{\Root{I}{i}}(\Root{I}{k}) & \mbox{if } \min\{i,j\}<k\leq \max\{i,j\}, \\\Root{I}{k} & \mbox{otherwise}.\end{array}\right.
    \]

    \item\label{it:weightflip} The map~$\Weight{J}{\cdot}$ is obtained from~$\Weight{I}{\cdot}$~by
    \[
      \Weight{J}{k} = \begin{cases} s_{\Root{I}{i}}(\Weight{I}{k}) & \text{if } \min(i,j) < k \le \max(i,j), \\ \Weight{I}{k} & \text{otherwise}. \end{cases}
    \]

    \item\label{it:weightroot} For $k \in I$, we have for $k' \in I$ that
    \[
      \dotprod{\Root{I}{k'}}{\Weight{I}{k}} = 0
    \]
    and we have for $k' \notin I$ that
    \[
      \begin{cases}
        \dotprod{\Root{I}{k'}}{\Weight{I}{k}} \geq 0 & \text{ if }k' \ge k, \\
        \dotprod{\Root{I}{k'}}{\Weight{I}{k}} \leq 0 & \text{ if }k' < k.
      \end{cases}
    \]
  \end{enumerate}
  \qed
\end{lemma} 

We first show that~\eqref{eq:root_coefficient} holds for the initial seed, and second that it is preserved under mutations.

\begin{proposition}\label{prop:extendedpartinitial}
  Let~$\greedy$ be the greedy facet of $\clusterComplex$.
  Then
  \[
    \big[\Mex(\greedy)\big]_{ji} = \Root{\greedy}{i}_j.
  \]
\end{proposition}
\begin{proof}
  This is the case as both sides are clearly equal to $\dotprod{\alpha_i}{\omega_j^\vee} = \delta_{i=j}$.
\end{proof}

\begin{proposition}\label{prop:exchangeRelationslemma}
Let~$I,J$ be two faces of $\clusterComplex$ with $I\setminus i = J\setminus j$, and let $k \in I\setminus i$ and $\ell \in \{1,\ldots,n\}$.
Then $\Root{J}{j} = -\Root{I}{i}$ and
\[
  \Root{J}{k}_\ell =
    \begin{cases}
      \Root{I}{k}_\ell + \Root{I}{i}_\ell \cdot \big[\Mpr(I)\big]_{ik} & \mbox{if } \Root{I}{i}_\ell \geq 0, \big[\Mpr(I)\big]_{ik}\geq 0, \\[5pt]
      \Root{I}{k}_\ell - \Root{I}{i}_\ell \cdot \big[\Mpr(I)\big]_{ik} & \mbox{if } \Root{I}{i}_\ell \leq 0, \big[\Mpr(I)\big]_{ik}\leq 0, \\[5pt]
      \Root{I}{k}_\ell & \mbox{otherwise}.
    \end{cases}
\]
\end{proposition}

\begin{proof}
  The property that $\Root{J}{j} = -\Root{I}{i}$ holds in general for facets $I \setminus i = J \setminus j$ in subword complexes.
  It is a direct consequence of \Cref{lem:lemma33}\eqref{it:rootflip}.

  It thus remains to show that $\Root{J}{k}_\ell$ is obtained from $\Root{I}{k}_\ell$ as described.
  For simplicity, observe that we can assume that $i < j$ as every facet of any subword complex $\subwordComplex(\Q)$ can be obtained from the greedy facet by a sequence of increasing flips.
  This implies, again by \Cref{lem:lemma33}\eqref{it:rootflip}, that $\Root{I}{i} \in \Phiplus$ and thus $\Root{I}{i}_\ell \geq 0$.
  Even though this is not needed, we note that the case of a decreasing flip $i > j$ could also be computed in the exact same way.

  The first case is $k \in \{i+1,\ldots,j-1\}$.
  It follows from~\cite[Lemma~6.43]{PS2015} that also $\big[\Mpr\big]_{ik} \geq 0$.
  And, as desired, we obtain
  \begin{align*}
    \Root{J}{k}_\ell
      &= \dotprod{\wordprod{Q}{[k]\setminus J} (\alpha_{q_k})}{\omega_\ell^\vee} \\
      &= \dotprod{\wordprod{Q}{[i]\setminus I} \cdot q_i \big( \wordprod{Q}{[i,k]\setminus I} (\alpha_{q_k})\big)}{\omega_\ell^\vee} \\
      &= \dotprod{\wordprod{Q}{[i]\setminus I} \cdot \Big(\wordprod{Q}{[i,k]\setminus I}(\alpha_{q_k}) - \dotprod{\wordprod{Q}{[i,k]\setminus I} (\alpha_{q_k})}{ \alpha_{q_i}^{\vee}} \alpha_{q_i}\Big)}{\omega_\ell^\vee} \\
      &= \dotprod{\wordprod{Q}{[k]\setminus I}(\alpha_{q_k})}{\omega_\ell^\vee} - \dotprod{\wordprod{Q}{[i,k]\setminus I} (\alpha_{q_k})}{ \alpha_{q_i}^{\vee}} \cdot \dotprod{\wordprod{Q}{[i]\setminus I}(\alpha_{q_i})}{\omega_\ell^\vee} \\
      &= \Root{I}{k}_\ell + \Root{I}{i}_\ell \cdot \big[\Mpr(I)\big]_{ik},
  \end{align*}
  where, as before, we write $\cw{c} = \q_1\cdots\q_{n+N}$.
  The first and the last equalities are the definitions together with \Cref{thm:exchangematrix}.
  The second equality is obtained as we do the flip from $i \in I$ to $j \in J$, the third equality is the definition of the application of the simple reflection~$q_i$ to $\wordprod{Q}{[i,k]\setminus I}(\alpha_{q_k})$, and the fourth equality is the linearity of the inner product.

  The second case is $k \notin \{i,\ldots,j\}$.
  It follows from~\cite[Lemma~6.43]{PS2015} that $\big[\Mpr\big]_{ik} \leq 0$, while $\Root{I}{i}_\ell \geq 0$.
  And indeed, the flip from~$i$ to~$j$ does not effect the root function at~$k$ by \Cref{lem:lemma33}\eqref{it:rootupdate}, and we obtain that $\Root{I}{k}_\ell = \Root{J}{k}_\ell$, as desired.
\end{proof}

We are now in the situation to deduce \Cref{thm:extendedpart}.
\begin{proof}[Proof of \Cref{thm:extendedpart}]
  It follows from \Cref{it:matrixmutation} and  \Cref{prop:exchangeRelationslemma} that
  \[
    \big[\Mex(I)\big]_{\ell k} = \Root{I}{k}_\ell \Longrightarrow \big[\Mex(J)\big]_{\ell k'} = \Root{J}{k'}_\ell
  \]
  for $I \setminus i = J \setminus j$ and either $(k,k') = (i,j)$ or $k = k' \neq i$.
  As \Cref{prop:extendedpartinitial} provides the equality for the initial mutation matrix, we obtain $\big[\Mex(I)\big]_{\ell i} = \Root{I}{i}_\ell$ for all $i \in I$.

  The property of the sign-coherence then follows as $\Roots{I} \subseteq \Phi$ for all facets~$I$, and the fact that $\Roots{I}$ forms a basis of the root space is a direct consequence of its iterative description.
\end{proof}

After this calculation, we present several general lemmas about cluster complexes that we then use in \Cref{sec:proofmain1} in type~$A_n$ to deduce \Cref{thm:main1}.

\begin{lemma}\label{lem:weightequality}
  Let $I,J$ be two facets of $\clusterComplex$ with $k \in I \cap J$.
  Then $\Weight{I}{k} = \Weight{J}{k}$.
\end{lemma}

\begin{proof}
  This is a direct consequence of \Cref{lem:lemma33}\eqref{it:rootupdate} and~\eqref{it:weightroot} and the observation that all facets of $\clusterComplex$ containing~$k$ are connected by flips (see \cite[Corollary~3.11]{PS2015}).
  (Indeed, this property of the weight function was already used in the proof of~\cite[Proposition~6.8]{PS2015}.)
\end{proof}

This lemma implies as well that $\Weight{I}{i} = \Weight{\antigreedy}{i}$ for any facet~$I$ and any $i \in I$.
If $i \in \antigreedy$, this follows immediate.
Otherwise, this follows from the observation that one can construct a facet $I'$ with $i \in I'$ being minimal.
\Cref{lem:weightequality} then implies that $\Weight{I}{i} = \Weight{I'}{i}$ and \Cref{lem:lemma33}\eqref{it:weightflip} implies that $\Weight{I'}{i} = \Weight{\antigreedy}{i}$.
We do not make further use of this additional information.

\medskip

We next recall the following lemma.

\begin{lemma}[{\cite[Lemma~4.5]{PS2015}}]\label{lem:positiveweightshift}
 Let $I \setminus i=J\setminus j$ with $i < j$ be two facets of $\clusterComplex$.
 For any $k \in \{1,\ldots,n+N\}$ we then have
 \[
   \Weight{J}{k} = \Weight{I}{k} - \lambda\Root{I}{i} \text{ for } \lambda\in\RR_{\geq 0} \text{ and } \Root{I}{i} \in \Phiplus.
 \]
 In particular, $\Weight{I}{k} - \Weight{J}{k} \in L^+$.
\end{lemma}
\begin{proof}
  This is a direct consequence of \Cref{lem:lemma33}\eqref{it:weightflip} and~\eqref{it:weightroot}.
\end{proof}

The following lemma has not been considered before and will serve as the starting point of understanding $F$-polynomials in terms of the weight function.

\begin{lemma}\label{lem:rootlemma}
  For $k \in \{n+1,\ldots,n+N\}$, we have that
  \[
    \Weight{\greedy}{k} - \Weight{\antigreedy}{k} = \Root{\greedy}{k}.
  \]
\end{lemma}

Observe that, as we have seen in Equation~\eqref{eq:rootbijection} on page~\pageref{eq:rootbijection}, this is also closely related to the bijection relating cluster algebras and subword complexes.

\begin{proof}[Proof of \Cref{lem:rootlemma}]
  Starting with the greedy facet~$\greedy$, we flip the first position as long as we can without flipping into position~$k$ to obtain a facet~$I$.
  Observe that along these flips, the facet always consists of a consecutive sequence of the~$n$ simple reflections.
  We therefore obtain, up to commutations of consecutive commuting letters, that $I = \{k-n,\ldots,k-1\}$ and
  \[
    \Weight{\greedy}{k} = \Weight{I}{k}.
  \]
  By the same argument, we flip the last position in $\antigreedy$ until we flip into position~$k$ to obtain a facet~$J$.
  Again up to commutations of consecutive commuting letters, we get that $J = \{k,\ldots,k+n-1\}$ and
  \[
    \Weight{\antigreedy}{k} = \Weight{J}{k} = \Weight{J'}{k}.
  \]
  Here~$J'$ is the facet obtained from~$J$ by again flipping in~$J$ the last position~$n-1$ times so that, up to commutations, we have $J' = \{k-n+1,\ldots,k\}$.
  The second equality is thus a direct consequence of \Cref{lem:weightequality} as~$k \in J \cap J'$.
  With these observations, we finally obtain for $\cwo{c} = \q_1\cdots\q_N$ that
  \begin{align*}
     \Weight{\greedy}{k} -  \Weight{\antigreedy}{k}
      &= \Weight{\{k-n,\ldots,k-1\}}{k} -  \Weight{\{k-n+1\ldots,k\}}{k} \\
      &= q_1\cdots q_{k-1}(\omega_{q_k}) - q_1\cdots q_{k}(\omega_{q_k}) \\
      &= q_1\cdots q_{k-1}(\omega_{q_k}) - q_1\cdots q_{k-1}(\omega_{q_k} - \alpha_{q_k}) \\
      &= q_1\cdots q_{k-1}(\alpha_{q_k}) \\
      &= \Root{\greedy}{k},
  \end{align*}
  as desired.
\end{proof}

We refer to the last table in \Cref{ex:subwordA2} for an example of this correspondence in type~$A_2$.

\section{Proof of \Cref{thm:main1}}\label{sec:proofmain1}

Before we recall the construction for the $F$-polynomials of type~$A_n$ to prove \Cref{thm:main1}, we set the needed notations.
The Coxeter group~$W$ is the symmetric group $\Sym_{n+1}$ acting on $\RR^{n+1} \big/ \RR(1,\ldots,1)$, whose simple system~$\sref$ are the set of simple transpositions $\sref=\{\tau_1,\ldots,\tau_n\}$ for~$\tau_i = (i,i+1)$ interchanging $e_i$ and $e_{i+1}$.
Thus, the Coxeter element~$c$ is given by the product of all simple transpositions in some order.
The simple roots are moreover given by $\Delta = \set{e_i - e_{i+1}}{1 \leq i \leq n}$, the positive roots by $\Phiplus = \set{ e_i - e_{j+1}}{1 \leq i \leq j \leq n}$, and the fundamental weights by $\nabla = \set{e_1+\ldots+e_i}{1 \leq i \leq n}$.
We refer to \Cref{ex:subwordA2} on page~\pageref{ex:subwordA2} for these notations in type~$A_2$.

For consecutive simple transpositions, we write~$\tau_i < \tau_{i-1}$ if~$\tau_i$ appears to the left of~$\tau_{i-1}$ in~$c$, and~$\tau_i > \tau_{i-1}$ if~$\tau_i$ appears to the right of~$\tau_{i-1}$.
We say that an element $\tau_{i_1} \cdots \tau_{i_m}$ is a \Dfn{prefix of~$c$} if there is a reduced word~$\c$ for~$c$ beginning with~$\tau_{i_1},\ldots,\tau_{i_m}$.
If all $\tau_{i_1},\ldots,\tau_{i_m}$ are inside the interval $\{\tau_i,\tau_{i+1},\ldots,\tau_j\}$ for $i \leq j$, we moreover say that it is a \Dfn{prefix of~$c$ restricted to $\{\tau_i,\ldots,\tau_j\}$} if the prefix property holds after removing all letters not in $\{\tau_i,\ldots,\tau_j\}$ from~$c$.
As an example, consider $c = \tau_1\tau_3\tau_2 = \tau_3\tau_1\tau_2 \in \Sym_{4}$.
The prefixes of~$c$ are $-,\tau_1,\tau_3,\tau_1\tau_3 = \tau_3\tau_1,\tau_1\tau_3\tau_2$, and the prefixes of~$c$ restricted to $\{\tau_1,\tau_2\}$ are $-,\tau_1,\tau_1\tau_2$.

\subsection{$F$-polynomials from $T$-paths}\label{sec:fpolystpaths}

R.~Schiffler derived in~\cite{S2008} an explicit formula for the cluster variables of type~$A_n$ via \emph{T-paths}.
These are certain path on the diagonals of triangulations of a regular $(n+3)$-gon.
G.~Musiker and R.~Schiffler then extended that description and obtained in~\cite{MS2010} an explicit formula for cluster variables in a similar fashion for cluster algebras with principal coefficients associated to unpunctured surfaces.
Together with L.~Williams, they extended in~\cite{MSW2011} the results also to arbitrary surfaces, allowing punctures.
In this section, we review that construction for type~$A_n$ to establish the needed notions to relate the description to the weight function in order to derive \Cref{thm:main1}.
To present their results in a convenient way, we follow~\cite[Section~5]{MS2010} as they directly work with principal coefficients, except that we use slightly simplified notions of $T$-paths.

\medskip

Let~$T$ be a triangulation of a regular $(n+3)$-gon, with boundary diagonals (or edges) labelled by $B_1,\ldots,B_{n+3}$ and with proper diagonals labelled by $\tau_1,\ldots,\tau_n$.
An example can be found in \Cref{fig:A4trianexample}(a); we use $A,B,C,\ldots$ instead of $B_1,B_2,B_3,\ldots$ and $1,2,3,\ldots$ instead of $\tau_1,\tau_2,\tau_3,\ldots$ in examples for better readability.

\begin{figure}[t]
  \begin{center}
    \begin{tabular}{ccc}
      \includegraphics[scale=.5]{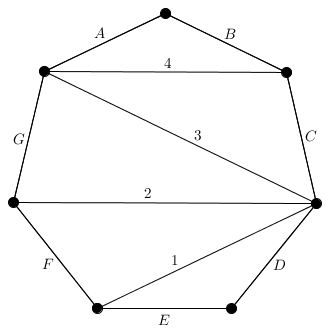}
      &&
      \includegraphics[scale=.5]{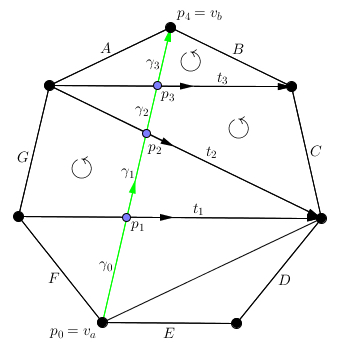} \\[3pt]
      (a) && (b)
    \end{tabular}
  \end{center}
  \caption{\label{fig:A4trianexample}A path in a triangulation of the $7$-gon.}
\end{figure}

Let $\gamma\notin T$ be another proper diagonal connecting non-adjacent vertices~$v_a$ and~$v_b$, oriented from~$v_a$ to~$v_b$.
Denote the intersection points of~$\gamma$ with the diagonals in~$T$ along its orientation by $p_1,\ldots,p_d$, and the corresponding diagonals in~$T$ by~$t_1,\ldots,t_d$.
Let $\gamma_{k}$ denote the segment of~$\gamma$ from point~$p_k$ to point~$p_{k+1}$, where we use $p_0=v_a$ and $p_{d+1}=v_b$.
Each $\gamma_k$ lies in exactly one triangle~$\triangle_k$, and we orient the diagonal~$t_k$ in~$T$ by the orientation induced from the counterclockwise orientation of~$\triangle_k$.
Note that if one considers the opposite path $\gamma^{-1}$ from~$v_b$ to~$v_a$, then the segment~$\gamma_i$ would become $\gamma_{d+1-i}$ and~$t_i$ would become $t_{d+1-i}$.
Moreover, the induced orientations of all~$t_k$ would change.

A \Dfn{$T$-path}~$\zeta$ from~$v_a$ to~$v_b$ in~$T$ is a path~$\zeta=(\zeta_1,\ldots,\zeta_{2d+1})$ in~$T$ which uses the oriented diagonals in the even positions.
In symbols,~$\zeta_{2k}=t_k$ for $1 \leq k \leq d$.
Observe that such a $T$-path is uniquely determined by the directions in which the diagonals~$t_1,\ldots,t_d$ in the even position are followed.
If the direction of~$\zeta$ coincides along the diagonal~$t_k$ with the direction induced by the counterclockwise orientation of the triangle~$\triangle_k$, we write that~$\zeta$ travels~$t_k$ in \Dfn{positive direction}, and it travels~$t_k$ in \Dfn{negative direction} otherwise.
It is not hard to see that there is always a unique $T$-path that travels all~$t_k$'s in positive direction.
We call this path the \Dfn{greedy $T$-path}, and denote it by~$\zetag$.
Similarly, we denote by~$\zetaag$ the \Dfn{antigreedy $T$-path} that travels all~$t_k$'s in negative direction.
(Note that these appeared in~\cite{MS2010} as $\widetilde\alpha_{P_+}$ and $\widetilde\alpha_{P_-}$ in the paragraph before Theorem~5.1.)
For instance, the greedy $T$-path in \Cref{fig:A4trianexample}(b) is $(F,2,3,3,3,4,B)$ and the antigreedy $T$-path is $(1,2,2,3,4,4,A)$.

We say that two $T$-paths~$\zeta$ and~$\zeta'$ are \Dfn{flipped} if~$\zeta$ and~$\zeta'$ only differ in two odd positions $2k-1$ and $2k+1$ for some~$k$.
In other words,~$t_k$ is the unique diagonal which is traveled by~$\zeta$ and~$\zeta'$ in opposite directions, while all others are traveled in the same direction.
We thus also say that~$t_k$ \Dfn{is flipped} between~$\zeta$ and~$\zeta'$.
In \Cref{fig:A4trianexample}(b), flipping~$\zeta_6 = t_3$ in the $T$-path $(F,2,3,3,3,4,B)$ yields the $T$-path $(F,2,3,3,C,4,A)$.

\medskip

To a $T$-path~$\zeta=(\zeta_1,\ldots,\zeta_{2d+1})$, one associates the monomial $m[\zeta]$ given by the product of variables $y_\ell$ such that~$\zeta_{2k} = t_k = \tau_\ell$ is traveled in positive direction.
For instance, the greedy $T$-path~$\zetag = (F,2,3,3,3,4,B)$ yields the monomial $m[{\zetag}] = y_2y_3y_4$, while the antigreedy $T$-path~$\zetaag = (1,2,2,3,4,4,A)$ yields $m[{\zetaag}] = 1$.
Moreover, all monomials obtained from $T$-paths for the diagonal~$\gamma$ in the example are given by
\[
  \begin{array}{|C{4px}C{8px}C{4px}C{8px}C{4px}C{8px}C{8px}|c|}
    \hline
    &&&\zeta&&&                 & m[\zeta]        \\[2pt]
    \hline
    &&&&&&&\\[-9pt]
    F&2^+&3&3^+&3&4^+&B & y_2y_3y_4 \\[2pt]
    F&2^+&3&3^+&C&4^-&A & y_2y_3    \\[2pt]
    1&2^-&G&3^+&3&4^+&B & y_3y_4    \\[2pt]
    1&2^-&G&3^+&C&4^-&A & y_3       \\[2pt]
    1&2^-&2&3^-&4&4^-&A & 1         \\[2pt]
    \hline
  \end{array}
\]
where we labelled the even steps $\zeta_{2k} = t_k$ for $1 \leq k \leq d$ with a~``$+$'' if it is traveled in positive direction and with a~``$-$'' if it is traveled in negative direction.
Also observe how $m[\zeta]$ changes under flips.
If~$t_i$ is flipped between~$\zeta$ and~$\zeta'$ then $m[\zeta] = m[\zeta']\cdot y_\ell$ if~$t_i = \tau_\ell$ is traveled in~$\zeta$ in positive direction (and thus~$\zeta'$ in negative direction).

\medskip

As we have noted above, the orientations of the~$t_k$'s depend on the orientation of~$\gamma$, while the monomial $m[\zeta]$ \emph{does not} depend on this orientation, but only on the two unordered endpoints $\{v_a,v_b\}$.
This combinatorial model now provides a description of the $F$-polynomials for the cluster algebra where the initial datum is the fixed given triangulation~$T$ of a regular $(n+3)$-gon.
It is well-known that $F$-polynomials for this cluster algebra are indexed by (unoriented) diagonals $\gamma \notin T$, see~\cite{MS2010}.

\begin{theorem}[{\cite[Theorem~5.1]{MS2010}}]\label{thm:schiffler}
  Let~$T$ be a triangulation of the regular $(n+3)$-gon, and let $\gamma \notin T$ with endpoints $\{v_a,v_b\}$.
  The $F$-polynomial $F_\gamma(\y)$ of~$\gamma$ is then given by
  \[
    F_\gamma(\y) = \sum m[\zeta],
  \]
  where the sum ranges over all $T$-paths~$\zeta$ from~$v_a$ to~$v_b$.
  \qed
\end{theorem}

\medskip
Next, we recall how to associate a triangulation~$T_c$ of the regular $(n+3)$-gon to a Coxeter element~$c$ in type~$A_n$.
\begin{enumerate}
  \item Pick a fixed vertex of the $(n+3)$-gon, labelled~$v_1$, and draw an edge connecting the two vertices adjacent to $v_1$.
  Label the new edge by $\tau_1$.

  \item For each $i = 2,\ldots,n$, label the vertex
    \[
      {}\begin{cases}
        \text{clockwise} & \text{ if } \tau_i < \tau_{i-1} \\
        \text{counterclockwise} & \text{ if } \tau_i > \tau_{i-1}
      \end{cases}
    \]
    from~$v_{i-1}$ by~$v_i$, draw an edge connecting the two vertices adjacent to~$v_i$ and different from~$v_{i-1}$, and label the new edge~$\tau_i$.
\end{enumerate}

The simple example of type $A_3$ with $c = \tau_1\tau_3\tau_2$ is
\begin{center}
  \includegraphics[scale=.3]{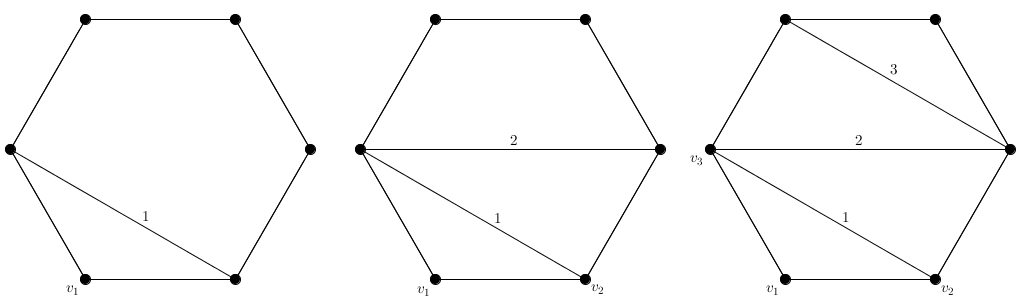}.
\end{center}
Moreover, the triangulation in \Cref{fig:A4trianexample} corresponds to the Coxeter element $c = \tau_3\tau_2\tau_1\tau_4$ in type~$A_4$.

\begin{lemma}\label{lem:diagonalspositiveroots}
  Let~$c$ be a Coxeter element in type~$A_n$, and let~$1 \leq i \leq j \leq n$.
  Then there is a unique diagonal~$\gamma \notin T_c$ that crosses exactly the diagonals labelled by $\tau_i,\tau_{i+1},\ldots,\tau_j$, and every diagonal not in~$T_c$ can be obtained this way.
\end{lemma}

\begin{proof}
  The triangulations that can be obtained (up to rotational symmetry) from a Coxeter element by the procedure are exactly the triangulations that do not have inner triangles, \ie, no triangles for which all three sides are proper diagonals.
  As the diagonals are labelled consecutively, the statement follows.
\end{proof}

\begin{figure}[t]
  \resizebox{\columnwidth}{!}{
    \begin{tabular}{|c||c|c|c|}
    \hline
    $\beta$ & $e_1-e_2$                                       & $e_3-e_4$          & $e_1-e_4$ \\ \hline
    \raisebox{-.9\height}{$\gamma$} & $\raisebox{-.9\height}{\includegraphics[scale=.25]{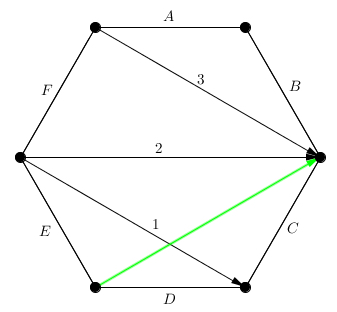}}$   & $\raisebox{-.9\height}{\includegraphics[scale=.25]{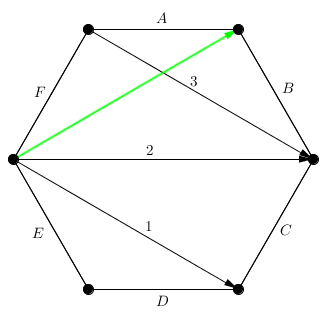}}$ & $\raisebox{-.9\height}{\includegraphics[scale=.25]{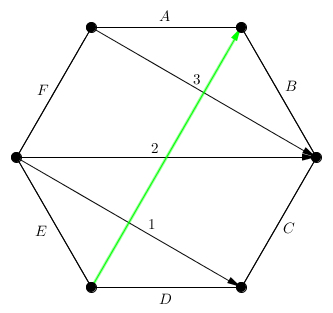}}$ \\ \hline
    $T_c$-paths   &   $\begin{array}{C{4px}C{8px}C{4px}C{8px}C{4px}C{8px}C{8px}}
      &&&\\[-9pt]
      E&1^+&C  \\[2pt]
      D&1^-&2     \\[2pt]
    \end{array}$  &    $\begin{array}{C{4px}C{8px}C{4px}C{8px}C{4px}C{8px}C{8px}}
      &&&\\[-9pt]
      2&3^+&A  \\[2pt]
      F&3^-&B     \\[2pt]
    \end{array}$    & $\begin{array}{C{4px}C{8px}C{4px}C{8px}C{4px}C{8px}C{8px}}
      &&&&&&\\[-9pt]
      E&1^+&1&2^+&3&3^+&B  \\[2pt]
      E&1^+&C&2^-&F&3^+&B     \\[2pt]
      D&1^-&2&2^-&F&3^+&B     \\[2pt]
      E&1^-&2&2^-&2&3^+&A        \\[2pt]
      D&1^-&2&2^-&2&3^-&A         \\[2pt]
    \end{array}$ \\ \hline
    $\Fpoly{\beta}{\y}$ & $y_1+1$ & $y_3+1$ & $y_1y_2y_3+y_1y_3+y_3+y_1+1$ \\ \hline\hline
    $\beta$ & $e_2-e_4 $                                       & $e_1-e_3$          & $e_2-e_3$ \\ \hline
    \raisebox{-.9\height}{$\gamma$}  & $\raisebox{-.9\height}{\includegraphics[scale=.25]{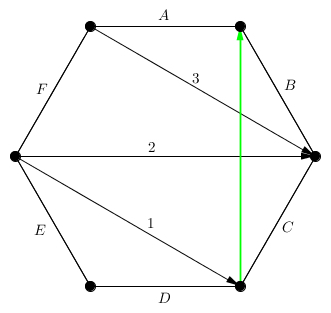}}$   & $\raisebox{-.9\height}{\includegraphics[scale=.25]{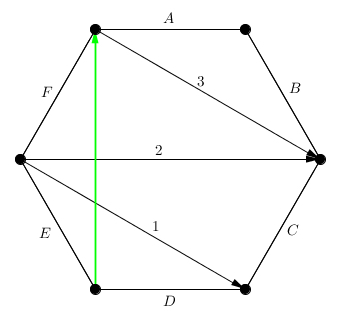}}$ & $\raisebox{-.9\height}{\includegraphics[scale=.25]{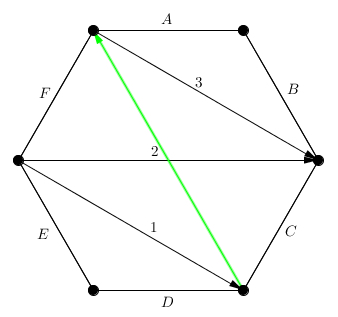}}$ \\ \hline
    $T_c$-paths   &   $\begin{array}{C{4px}C{8px}C{4px}C{8px}C{4px}C{8px}C{8px}}
      &&&&\\[-9pt]
      1&2^+&3&3^+&B  \\[2pt]
      C&2^-&F&3^+&B     \\[2pt]
      C&2^-&2&3^-&A     \\[2pt]
    \end{array}$  &    $\begin{array}{C{4px}C{8px}C{4px}C{8px}C{4px}C{8px}C{8px}}
      &&&&\\[-9pt]
      E&1^+&1&2^+&3  \\[2pt]
      E&1^+&C&2^-&F    \\[2pt]
      D&1^-&2&2^-&F    \\[2pt]
    \end{array}$    & $\begin{array}{C{4px}C{8px}C{4px}C{8px}C{4px}C{8px}C{8px}}
      &&&\\[-9pt]
      1&2^+&3  \\[2pt]
      C&2^-&F     \\[2pt]
    \end{array}$ \\ \hline
    $\Fpoly{\beta}{\y}$ & $y_2y_3+y_3+1$ & $y_1y_2+y_1+1$ & $y_2+1$ \\ \hline
    \end{tabular}
  }
  \caption{\label{fig:132Tpaths}$T_c$-paths for $c=\tau_1\tau_3\tau_2$ in type $A_3$.}
\end{figure}

\begin{example}
  \Cref{fig:132Tpaths} shows for $c=\tau_1\tau_3\tau_2$ in type~$A_3$ and each positive root $\beta=e_i-e_{j+1}$, the unique diagonal $\gamma\notin T_c$ crossing the diagonals $\tau_i,\ldots,\tau_j$, and all $T_c$-paths for this~$\gamma$.
\end{example}

We have the following corollary of the above \Cref{thm:schiffler}, which we will then use to deduce \Cref{thm:main1}.

\begin{corollary}\label{cor:TpathAn}
  Let~$c$ be a Coxeter element in type~$A_n$, let $\beta = e_i - e_{j+1}$ be a positive root, and let~$\gamma \notin T_c$ be the unique diagonal crossing exactly the diagonals labelled $\tau_i,\ldots,\tau_j$ in~$T_c$.
  Let the endpoints of~$\gamma$ be $\{v_a,v_b\}$.
  The $F$-polynomials for the cluster algebra $\Alg(W,c)$ associated to~$\beta$ is then given by
  \[
    \Fpoly{\beta}{\y} = \sum m[\zeta],
  \]
  where the sum ranges over all $T$-paths~$\zeta$ from~$v_a$ to~$v_b$.
\end{corollary}

\begin{proof}
  This follows from the well known connection between $\Alg(W,c)$ and the triangulation~$T_c$ described above.
\end{proof}

\begin{example}
For $c=\tau_1\tau_3\tau_2$ in type $A_3$, \Cref{fig:132Tpaths} also provides all $\Fpoly{\beta}{\y}$ for~$\beta \in\Phiplus$ using the construction of $T_c$-paths.
\end{example}

We finally need the following proposition regarding possible flips in triangulations with respect to a Coxeter element~$c$.

\begin{proposition}\label{prop:greedytpath}
  Let~$T_c$ be the triangulation associated to a Coxeter element~$c$, and let~$\gamma\notin T_c$ be the unique diagonal oriented from $v_a$ to $v_b$ which crosses exactly the diagonals $\tau_i,\ldots,\tau_j$ in this order.
  Then for any prefix $\tau_{i_1} \cdots \tau_{i_m}$ of~$c$ restricted to $\{\tau_i,\ldots,\tau_{j}\}$, one can flip the diagonals labelled~$\tau_{i_m},\ldots,\tau_{i_1}$ in this order in the antigreedy $T_c$-path from~$v_a$ to~$v_b$.
  Moreover, every $T_c$-path from~$v_a$ to~$v_b$ is obtained this way for a unique prefix.
\end{proposition}

\begin{proof}
  We explicitly describe the four possible restrictions for directions in which $T_c$-paths can travel.
  To this end, consider the situation that~$t_{i-1}$ and~$t_{i}$ are oriented towards their shared vertex in~$T_c$, or, equivalently, that $\tau_{i} < \tau_{i-1}$ (see the path for $e_2-e_4$ in \Cref{fig:132Tpaths}).
  Then, any~$T_c$-path $\zeta$ from~$v_a$ to~$v_b$
  \begin{itemize}
    \item that travels the diagonal~$\zeta_{2i-2} = t_{i-1}$ in \emph{positive} direction must also travel $\zeta_{2i} = t_{i}$ in \emph{positive} direction, and
    \item that travels the diagonal~$\zeta_{2i} = t_i$ in \emph{negative} direction must also travel $\zeta_{2i-2} = t_{i-1}$ in \emph{negative} direction.
  \end{itemize}
  The situation where~$t_{i-1}$ and~$t_{i}$ are oriented away from their shared vertex in~$T_c$, or, equivalently, that $\tau_i > \tau_{i-1}$ is the same with the roles of positive and negative direction interchanged (see the path for $e_1-e_3$ in \Cref{fig:132Tpaths}).

  Clearly, these are the only restrictions on $T_c$-paths.
  This means that a given sequence of orientations of the $t_i$'s corresponds to a $T_c$-path if and only if these restictions are satisfied.
  It then directly follows that every $T_c$-path is uniquely obtained from the antigreedy $T_c$-path~$\zetaag$ (which travels all the diagonals in negative direction) by flipping diagonals labelled $\tau_{i_m},\ldots,\tau_{i_1}$ in this order for a prefix $\tau_{i_1} \cdots \tau_{i_m}$ of~$c$ restricted to $\{\tau_i,\ldots,\tau_{j}\}$, as desired.
  Observe here that if one considers two different words for the same prefix, then both sequences of flips yield the same $T_c$-path, as expected.
\end{proof}

We refer to \Cref{fig:132Tpaths} for several examples, and also to \Cref{ex:A3two} below for two concrete computations.

\begin{corollary}\label{cor:TpathAn2}
  In the situation of \Cref{cor:TpathAn}, we have that
  \[
    \Fpoly{\beta}{\y} = \sum y_{i_1}\cdots y_{i_m}
  \]
  where the sum ranges over all prefixes $\tau_{i_1} \cdots \tau_{i_m}$ of~$c$ restricted to $\{\tau_i,\ldots,\tau_{j}\}$.
\end{corollary}

\begin{proof}
  \Cref{prop:greedytpath} implies that $\Fpoly{\beta}{\y} = \sum m[\zeta_{i_1,\ldots,i_m}]$ where the sum ranges over all prefixes of~$c$ restricted to $\{\tau_i,\ldots,\tau_j\}$.
  Moreover, the resulting $T$-path $\zeta_{i_1,\ldots,i_m}$ that is obtained from the antigreedy $T$-path $\zetaag$ by flipping the diagonals $\tau_{i_m},\ldots,\tau_{i_1}$ in this order.
  The definition of $m[\zeta]$ thus implies that $m[\zeta_{i_1,\ldots,i_m}] = y_{i_1}\cdots y_{i_m}$, as desired.
\end{proof}

\begin{example}
\label{ex:A3two}
  Following the two examples of the two positive roots $e_1-e_4 = \alpha_1+\alpha_2+\alpha_3$ and $e_1-e_3 = \alpha_1+\alpha_2$ in type~$A_3$ with $c=\tau_1\tau_3\tau_2$ above, we obtain that the prefixes of~$c$ are $-,\tau_1,\tau_3,\tau_1\tau_3,\tau_1\tau_3\tau_2$ yielding 
  \[
    \Fpoly{e_1-e_4}{\y} = 1 + y_1 + y_3 + y_1y_3 + y_1y_2y_3,
  \]
  and that the prefixes of~$c$ restricted to $\{\tau_1,\tau_2\}$ are $-,\tau_1,\tau_1\tau_2$ yielding
  \[
    \Fpoly{e_1-e_3}{\y} = 1 + y_1 + y_1y_2.
  \]
  Both cases can be checked in \Cref{fig:132Tpaths}.
\end{example}

\subsection{$F$-polynomials from subword complexes}

We now use \Cref{cor:TpathAn2} to obtain the $F$-polynomials from the weight vectors of $\clusterComplex$ by providing the analogous property of the weight vectors in \Cref{thm:propweights}.
We start with the following explicit description of all different weights $\Weight{I}{k}$ that occur for the various facets for a fixed position~$k$.
Given a positive root $e_i - e_{j+1}$ with $1 \leq i \leq j \leq n$.
We have seen in \Cref{lem:lemma33} that there is a unique~$k \in \{n+1,\ldots,n+N\}$ such that $\Root{\greedy}{k} = e_i - e_{j+1}$, and we have then seen in \Cref{lem:rootlemma} that
\begin{align}
  \Weight{\greedy}{k} - \Weight{\antigreedy}{k} = \Root{\greedy}{k} = e_i - e_{j+1}. \label{eq:weighttypeA}
\end{align}

This yields the following lemma.

\begin{lemma}\label{lem:greedyantigreedyweight}
  We have
  \begin{align}
    \Weight{\greedy}{k}     &= (\epsilon_1,\ldots,\epsilon_{i-1},\ 1,\ \epsilon_{i+1},\ldots,\epsilon_{j},\ 0,\ \epsilon_{j+2},\ldots,\epsilon_{n+1}) \label{eq:greedyAn}\\
    \Weight{\antigreedy}{k} &= (\epsilon_1,\ldots,\epsilon_{i-1},\ 0,\ \epsilon_{i+1},\ldots,\epsilon_{j},\ 1,\ \epsilon_{j+2},\ldots,\epsilon_{n+1}) \label{eq:antigreedyAn}
  \end{align}
  for fixed $\epsilon_i \in \{0,1\}$ with $i \in \{1,\ldots,n+1\}\setminus\{i,j+1\}$.
\end{lemma}

\begin{proof}
  This follows from~\eqref{eq:weighttypeA} as all weights in type~$A_n$ are $(0,1)$-vectors.
\end{proof}

Using this observation, one explicitly obtains these weights for a given Coxeter element~$c$ as described next.

\begin{proposition}\label{prop:weightsAn}
  In the situation of \Cref{lem:greedyantigreedyweight}, we have the following properties of the $\epsilon_i$'s (where the cases in~\eqref{eq:innerepsilon} are only considered either if $i>1$, or if $j<n$, respectively):
  \begin{enumerate}[(i)]
    \item $\epsilon_1 = \ldots = \epsilon_{i-1}$ and $\epsilon_{j+2} = \ldots = \epsilon_{n+1}$.
    \item\label{eq:innerepsilon} $\epsilon_{i-1} = \begin{cases}
                         0 &\text{ if } \tau_{i-1} < \tau_{i  } \\
                         1 &\text{ if } \tau_{i-1} > \tau_{i  }
                       \end{cases}$\quad and $\epsilon_{j+2} = \begin{cases}
                         0 &\text{ if }\tau_{j+1} < \tau_{j+2} \\
                         1 &\text{ if }\tau_{j+1} > \tau_{j+2}
                       \end{cases}$.
    \item for $i<\ell\leq j$, $\epsilon_\ell = \begin{cases}
                                        1 \text{ if }\tau_{\ell-1} < \tau_\ell \\
                                        0 \text{ if }\tau_{\ell-1} > \tau_\ell
                                      \end{cases}$.

  \end{enumerate}
\end{proposition}

\begin{proof}
  One can flip the letters in the initial copy of $\c$ inside $\cw{c}$ from right to left to obtain a sequence of increasing flips from~$\greedy$ to~$\antigreedy$ (which is actually the unique shortest chain of increasing flips from the greedy to the antigreedy facet).
  As the root configuration of~$\greedy$ is given by all simple roots $\{e_1-e_2,\ldots,e_n-e_{n+1}\}$, \Cref{lem:lemma33}\eqref{it:rootupdate} implies that along the above sequence of flips from~$\greedy$ to~$\antigreedy$, every pair~$\epsilon_\ell$ and $\epsilon_{\ell+1}$ of consecutive indices of $\Weight{\greedy}{k}$ is updated exactly once (where we set $\epsilon_i = 1$ and $\epsilon_{j+1} = 0$ as in~\eqref{eq:greedyAn}.
  Moreover, \Cref{lem:positiveweightshift} implies that along this procedure, either $\epsilon_\ell = \epsilon_{\ell+1}$ and the application of~$\tau_\ell$ does not change the weight, or $\epsilon_\ell = 1$ and $\epsilon_{\ell+1} = 0$, and this application moves the $1$ in position~$\ell$ into position~$\ell+1$.
  As such a move is not reversible again by \Cref{lem:positiveweightshift}, we directly obtain the second property.
  The first property is obtained with the additional observation that those entries coincide in~$\greedy$ and in~$\antigreedy$.
  The last property finally follows with the observation that every $\tau_\ell$ for $i \leq \ell \leq j$ must indeed move a~$1$ one position to the right to obtain~$\antigreedy$ from~$\greedy$ this way.
\end{proof}

\begin{example}
  We again consider $c=\tau_1\tau_3\tau_2$ in type $A_3$.
  \Cref{weightsc132A3} shows all weights $\Weight{I}{k}$ in this case, the weights for $\greedy$ and $\antigreedy$ can be computed as described in \Cref{lem:greedyantigreedyweight} and \Cref{prop:weightsAn}.
  \begin{figure}[t]
    \resizebox{\columnwidth}{!}{
      \begin{tabular}{|cc||c|c|c|c|c|c|}
        \hline
        && $\alpha_1$ & $\alpha_3$ & $\scriptstyle{\alpha_1+\alpha_2+\alpha_3}$ & $\alpha_2+\alpha_3$ & $\alpha_1+\alpha_2$ & $\alpha_2$ \\
        & $I$ & $4$ & $5$ & $6$ & $7$ & $8$ & $9$ \\
        \hline
        $\greedy=$\hspace*{-5pt}
        & \{1, 2, 3\} & $(1, 0, 0, 0)$ & $(1, 1, 1, 0)$ & $(1, 1, 0, 0)$ & $(0, 1, 0, 0)$ & $(1, 1, 0, 1)$ & $(0, 1, 0, 1)$ \\
        & \{1, 2, 9\} & $(1, 0, 0, 0)$ & $(1, 1, 1, 0)$ & $(1, 0, 1, 0)$ & $(0, 0, 1, 0)$ & $(1, 0, 1, 1)$ & $(0, 0, 1, 1)$ \\
        & \{1, 3, 5\} & $(1, 0, 0, 0)$ & $(1, 1, 0, 1)$ & $(1, 1, 0, 0)$ & $(0, 1, 0, 0)$ & $(1, 1, 0, 1)$ & $(0, 1, 0, 1)$ \\
        & \{1, 5, 7\} & $(1, 0, 0, 0)$ & $(1, 1, 0, 1)$ & $(1, 0, 0, 1)$ & $(0, 0, 0, 1)$ & $(1, 1, 0, 1)$ & $(0, 1, 0, 1)$ \\
        & \{1, 7, 9\} & $(1, 0, 0, 0)$ & $(1, 1, 0, 1)$ & $(1, 0, 0, 1)$ & $(0, 0, 0, 1)$ & $(1, 0, 1, 1)$ & $(0, 0, 1, 1)$ \\
        & \{2, 3, 4\} & $(0, 1, 0, 0)$ & $(1, 1, 1, 0)$ & $(1, 1, 0, 0)$ & $(0, 1, 0, 0)$ & $(1, 1, 0, 1)$ & $(0, 1, 0, 1)$ \\
        & \{2, 4, 8\} & $(0, 1, 0, 0)$ & $(1, 1, 1, 0)$ & $(0, 1, 1, 0)$ & $(0, 1, 0, 0)$ & $(0, 1, 1, 1)$ & $(0, 1, 0, 1)$ \\
        & \{2, 8, 9\} & $(0, 1, 0, 0)$ & $(1, 1, 1, 0)$ & $(0, 1, 1, 0)$ & $(0, 0, 1, 0)$ & $(0, 1, 1, 1)$ & $(0, 0, 1, 1)$ \\
        & \{3, 4, 5\} & $(0, 1, 0, 0)$ & $(1, 1, 0, 1)$ & $(1, 1, 0, 0)$ & $(0, 1, 0, 0)$ & $(1, 1, 0, 1)$ & $(0, 1, 0, 1)$ \\
        & \{4, 5, 6\} & $(0, 1, 0, 0)$ & $(1, 1, 0, 1)$ & $(0, 1, 0, 1)$ & $(0, 1, 0, 0)$ & $(1, 1, 0, 1)$ & $(0, 1, 0, 1)$ \\
        & \{4, 6, 8\} & $(0, 1, 0, 0)$ & $(1, 1, 0, 1)$ & $(0, 1, 0, 1)$ & $(0, 1, 0, 0)$ & $(0, 1, 1, 1)$ & $(0, 1, 0, 1)$ \\
        & \{5, 6, 7\} & $(0, 1, 0, 0)$ & $(1, 1, 0, 1)$ & $(0, 1, 0, 1)$ & $(0, 0, 0, 1)$ & $(1, 1, 0, 1)$ & $(0, 1, 0, 1)$ \\
        & \{6, 7, 8\} & $(0, 1, 0, 0)$ & $(1, 1, 0, 1)$ & $(0, 1, 0, 1)$ & $(0, 0, 0, 1)$ & $(0, 1, 1, 1)$ & $(0, 1, 0, 1)$ \\
        $\antigreedy=$\hspace*{-5pt} 
        & \{7, 8, 9\} & $(0, 1, 0, 0)$ & $(1, 1, 0, 1)$ & $(0, 1, 0, 1)$ & $(0, 0, 0, 1)$ & $(0, 1, 1, 1)$ & $(0, 0, 1, 1)$ \\
        \hline
      \end{tabular}
    }
    \caption{\label{weightsc132A3}The weights of the facets appearing in the unique shortest chain of increasing flips from the greedy to the antigreedy facet for $c=\tau_1\tau_3\tau_2$ in type $A_3$.}
  \end{figure}
\end{example}

To state the main observation towards the proof of \Cref{thm:main1}, define the following set of weights as an ``interval'' in the weights,
\[
  \IntervalWeight{k} = \set{\omega \in W(\nabla)}{\Weight{\greedy}{k} - \omega, \omega - \Weight{\antigreedy}{k} \in L^+}.
\]
Similarly, we define the following positive elements as an ``interval'' in the root space,
\begin{align*}
  \IntervalRoot{k}  &= \set{\gamma \in L^+}{\Weight{\greedy}{k}-\Weight{\antigreedy}{k}-\gamma \in L^+} \\
                    &= \set{\gamma \in L^+}{\hspace*{26.5pt}\Root{\greedy}{k}\hspace*{26.5pt}-\gamma \in L^+},
\end{align*}
and observe that
\[
  \IntervalRoot{k} = \bigset{ \omega - \Weight{\antigreedy}{k}}{ \omega \in \IntervalWeight{k}}.
\]
Using this notion, we deduce the following theorem from \Cref{prop:weightsAn}.

\begin{theorem}\label{thm:propweights}
  For any $k \in \{1,\ldots,n+N\}$, we have that $\IntervalWeight{k}$ equals the set of weights $\Weight{I}{k}$ for which the facet~$I$ can be obtained from the antigreedy facet $\antigreedy$ by flipping the letters $i_m,\ldots,i_1$ in this order for prefixes $\tau_{i_1} \cdots \tau_{i_m}$ of~$c$.
\end{theorem}

\begin{proof}
  Move from the weight given in~\eqref{eq:antigreedyAn} to the weight given in~\eqref{eq:greedyAn} inside the interval $\IntervalWeight{k}$ means that one can interchange any two consecutive positions $\epsilon_\ell,\epsilon_{\ell+1}$ whenever $\ell \in \{i,\ldots j\}$ and $(\epsilon_\ell,\epsilon_{\ell+1}) = (0,1)$.
  One directly observes that flipping letters $i_m,\ldots,i_1$ in this order in $\antigreedy$ for prefixes $\tau_{i_1} \cdots \tau_{i_m}$ of~$c$ yields the same set of vectors.
\end{proof}

Given this theorem, we obtain indeed all weights of facets of $\clusterComplex$.

\begin{corollary}\label{cor:intervalequality}
  We have for $k \in \{1,\ldots,n+N\}$ that
  \[
    \set{\Weight{I}{k}}{I \text{ facet of }\clusterComplex} = \IntervalWeight{k}.
  \]
\end{corollary}

\begin{proof}
  The inclusion
  \[
    \set{\Weight{I}{k}}{I \text{ facet of }\clusterComplex} \subseteq \IntervalWeight{k}
  \]
  is a direct consequence of \Cref{lem:positiveweightshift} as every facet lies on a sequence of increasing flips from the greedy to the antigreedy facet.
  The other direction is \Cref{thm:propweights}.
\end{proof}

\begin{proof}[Proof of \Cref{thm:main1}]
  Let $\Alg(W,c)$ be the cluster algebra of type~$A_n$ for a given Coxeter element~$c$, and let $\beta = e_i - e_{j+1}$ be a positive root.
  By \Cref{cor:TpathAn2}, we have that
  \[
    \Fpoly{\beta}{\y} = \sum y_{i_1}\cdots y_{i_m}
  \]
  where the sum ranges over all prefixes $\tau_{i_1} \cdots \tau_{i_m}$ of~$c$ restricted to $\{\tau_i,\ldots,\tau_{j}\}$.

  For the unique index~$k$ such that $\Root{\greedy}{k} = \beta$, the description of the interval $\IntervalWeight{k}$ in \Cref{thm:propweights} yields that the set of all $\Weight{I}{k} - \Weight{\antigreedy}{k}$ for which~$I$ is obtained from the antigreedy facet $\antigreedy$ by flipping the letters $i_m,\ldots,i_1$ in this order for prefixes $\tau_{i_1} \cdots \tau_{i_m}$ of~$c$ exactly matches the exponent vectors of the monomials $y_{i_1}\cdots y_{i_m}$ in the above sum.
  Together with \Cref{cor:intervalequality}, this implies \Cref{conj:newtonpolytope} for $\Alg(W,c)$.

  \Cref{conj:latticepoints} is then a direct consequence as all exponent vectors of monomials in $\Fpoly{\beta}{\y}$ are $(0,1)$-vectors.
\end{proof}

The given description of the $F$-polynomials in type~$A_n$ has the consequence that there is a generalization of Loday's realization of the classical associahedron mentioned in \Cref{rem:postnikov1} to all $c$-associahedra of type~$A_n$.

\begin{corollary}
  The type~$A_n$ $c$-associahedron is given by
  \[
    \sum \conv\{ e_{i_1} + \ldots + e_{i_m} \}
  \]
  where the sum ranges over all pairs $1 \leq i \leq j \leq n$ and each convex hull is over all prefixes $i_m \cdots i_1$ of the Coxeter element~$c$ restricted to $\{\tau_i,\ldots,\tau_{j}\}$.
\end{corollary}

We remark that C.~Lange obtained in~\cite{Lan2013} a different Minkowski decomposition of the $c$-associahedron into sums and differences of simplicies.
We refer to~\cite[Theorem~4.3]{Lan2013} and also to~\cite[Section~4]{LP2013} for details on that decomposition.

\begin{example}
  We have seen that the prefixes of $c = \tau_1\tau_3\tau_2$ are given by
  \[
    -,\tau_1,\tau_3,\tau_1\tau_3,\tau_1\tau_3\tau_2,
  \]
  so that we obtain that the associahedron is given by
  \begin{align*}
     & \conv\{000,100\} + \conv\{000,100,110\} + \conv\{000,100,001,101,111\} \\
    +& \conv\{000,010\} + \conv\{000,001,011\} + \conv\{000,001\}
  \end{align*}
  where the summands correspond to the intervals $\{\tau_i,\ldots,\tau_j\}$ for $1 \leq i \leq j \leq 3$ in the order $\tau_1, \tau_1\tau_2, \tau_1\tau_2\tau_3, \tau_2, \tau_2\tau_3, \tau_3$.
  For example, the second summand is given by the convex hull of $\{000,100,110\}$.
  These are the different indicator vectors of prefixes $-,\tau_1,\tau_1\tau_2$ of~$c$ where the letter~$\tau_3$ is deleted.
\end{example}

\bibliographystyle{alpha}
\bibliography{BrodskyStump}

\end{document}